\documentclass[10pt]{amsart}
\usepackage{verbatim}
\usepackage{eucal,url,amssymb,stmaryrd,enumerate,amscd}
\usepackage{amsmath}
\usepackage[pagebackref,colorlinks=true,linkcolor=blue,citecolor=blue]{hyperref}
\usepackage{amsfonts}
\usepackage{amsthm}
\usepackage[margin=1in]{geometry}
\numberwithin{equation}{section}
\newtheorem{thrm}{Theorem}[section]
\newtheorem{lemma}[thrm]{Lemma}

\newtheorem{cor}[thrm]{Corollary}
\newtheorem{dfn}[thrm]{Definition}

\newtheorem{rmrk}[thrm]{Remark}

\newtheorem{conv}[thrm]{Convention}

\newcommand{\Hn}{\mathbb{H}^n}

\newcommand{\QH}{\boldsymbol {G\,(\mathbb{H})}}

\newcommand{\abs}[1]{\lvert #1 \rvert}

\newcommand{\e}{\textbf {e}}

 1

\def\Re{{\frak R}{\frak e}\,}

\def\gr{\nabla f}
\def\bi{\nabla}

\begin{document}

\begin{abstract}

We prove quaternionic contact versions of two of  Obata's sphere theorems. On a compact quaternionic contact (qc) manifold of dimension bigger than seven and satisfying a Lichnerowicz type lower bound estimate we show that if the first positive eigenvalue of the sub-Laplacian takes the smallest possible value then, up to a homothety of the qc structure, the manifold is qc equivalent to the standard 3-Sasakian  sphere. The same conclusion is shown to hold on a non-compact qc manifold which is complete with respect to the associated Riemannian metric assuming the existence of a function with traceless horizontal Hessian. The third result of the paper is a qc version of Liouville's theorem showing that a qc-conformal diffeomorphism between open connected sets of the 3-Sasakian sphere is a restriction of an element of the qc-conformal automorphism group of the sphere.
\end{abstract}

\title[The Obata sphere theorems on a quaternionic contact manifold]{The
Obata sphere theorems on a quaternionic contact manifold of dimension bigger than seven}
\date{\today }
\author{S. Ivanov}
\address[Stefan Ivanov]{University of Sofia, Faculty of Mathematics and
Informatics, blvd. James Bourchier 5, 1164, Sofia, Bulgaria}
\email{ivanovsp@fmi.uni-sofia.bg}
\author{A. Petkov}
\address[Alexander Petkov]{University of Sofia, Faculty of Mathematics and
Informatics, blvd. James Bourchier 5, 1164, Sofia, Bulgaria}
\email{a\_petkov\_fmi@abv.bg}
\author{D. Vassilev}
\address[Dimiter Vassilev]{ Department of Mathematics and Statistics\\
University of New Mexico\\
Albuquerque, New Mexico, 87131-0001}
\email{vassilev@math.unm.edu}
\maketitle
\tableofcontents

\setcounter{tocdepth}{2}

\section{Introduction}

Motivated by the classical Lichnerowicz \cite{Li} and Obata \cite{O3} theorems,  earlier papers of the authors \cite{IPV1,IPV2} established a Lichnerowcz type lower bound estimate for the first eigenvalue of the sub-Laplacian on a compact quaternionic contact (qc) manifold. The case of equality in the lower bound estimate (Obata-type theorem)  was settled  in the special case of a 3-Sasakian compact manifold where it was shown that the lower bound {for the first eigenvalue of the sub-Laplacian} is achieved  if and only if  the 3-Sasakian manifold is isometric  to the standard 3-Sasakian sphere. {Quaternionic contact (qc) structures were introduced by O. Biquard \cite{Biq1} and are modeled on  the conformal boundary at infinity of the quaternionic hyperbolic space. {Thus, manifolds equipped with a  qc structure are examples of  sub-Riemannian geometries. The (locally)  3-Sasakian manifolds  were  characterized in \cite{IMV, IV1} by the vanishing of the torsion tensor of the Biquard connection.}
The qc geometry was a crucial geometric tool in finding the extremals and the best constant in the $L^2$ Folland-Stein Sobolev-type embedding, \cite{F2,FS},  completely described   on
the  quaternionic Heisenberg groups,  \cite{IMV1,IMV2}. }

In this paper we prove the full qc
version of Obata's results  for a general qc manifold of dimension bigger than seven. {We   find that  the equality case of
Lichnerowcz' type inequality on a compact qc manifold  of dimension at least eleven can be achieved only on the 3-Sasakian spheres. More general, we} show that on a complete
with respect to the associated Riemannian metric qc manifold a
certain (horizontal) Hessian equation, cf. \eqref{eq7}, allows a non-trivial solution if and
only if the manifold is qc homothetic to the standard 3-Sasakian sphere.

The qc seven dimensional case was considered in \cite{IPV2}, however, the
general qc Obata results in dimension seven remain open.

Turning to some details, let us recall the mentioned classical results. Using
the classical Bochner-Weitzenb\"ock formula Lichnerowcz \cite{Li} showed
that on a compact Riemannian manifold $(M,h)$ of dimension $n$ for which the
Ricci curvature is greater than or equal to that of the round unit $n$%
-dimensional sphere $S^n(1)$, i.e., $
Ric(X,Y)\geq (n-1)h(X,Y) $
 the first positive eigenvalue $\lambda_1$ of the (positive) Laplace operator
is greater than or equal to the first eigenvalue of the sphere, $%
\lambda_1\geq n.$ Subsequently Obata \cite{O3} proved that equality is
achieved if and only if the Riemannian manifold is isometric to $S^n(1)$ by
noting that the trace-free part of the Riemannian Hessian of an
eigenfunction $f$ with eigenvalue $\lambda=n$ vanishes, i.e., it satisfies
the system
\begin{equation}  \label{clasob}
(\nabla^h)^2 f = -fh
\end{equation}
after which he defines an isometry using analysis based on the geodesics and
Hessian comparison of the distance function from a point.

In fact, Obata showed that
on a complete Riemannian manifold $(M,h)$ equation \eqref{clasob} allows a
non-constant solution if and only if the manifold is isometric to the unit
sphere. In this case, the eigenfunctions corresponding to the first
eigenvalue are the solutions of \eqref{clasob}. 
Later, Gallot \cite{Gallot79} generalized these results to statements
involving the higher eigenvalues and corresponding eigenfunctions of the
Laplace operator.

The interest in relations between the spectrum of the Laplacian and
geometric quantities justified the interest in Lichnerowicz-Obata type
theorems in other geometric settings such as Riemannian foliations (and the
eigenvalues of the basic Laplacian) \cite{LR98,LR02}, \cite{JKR11} and \cite%
{PP11}, to CR geometry (and the eigenvalues of the sub-Laplacian) \cite{Gr},
\cite{Bar}, \cite{CC07,CC09a,CC09b}, \cite{ChW}, \cite{Chi06}, \cite{LL}, and to general
sub-Riemannian geometries, see \cite{Bau2} and \cite{Hla}. In the CR case,
Greenleaf \cite{Gr} gave a version of Lichnerowicz' result showing that if a
compact strongly pseudo-convex CR manifold $M$ of dimension $2n+1$, $n\geq 3$
satisfies a Lichnerowicz type inequality
\begin{equation*}
Ric(X,Y)+ 4A(X,JY)\geq (n+1) g(X,X)
\end{equation*}
for all horizontal vectors $X$, where $Ric$ and $A$ are, correspondingly,
the Ricci curvature and the Webster torsion of the Tanaka-Webster connection
(in the notation from \cite{IVZ,IV2}), then the first positive eigenvalue $%
\lambda_1$ of the sub-Laplacian satisfies the inequality $\lambda_1\geq {n}$%
. The standard (Sasakian) CR structure on the sphere achieves equality in
this inequality. Following \cite{Gr} the above cited results on a compact CR
manifold focused on adding a corresponding inequality for $n=1,2$ or
characterizing the equality case mainly in the vanishing Webster-torsion
case (the Sasakian case). The general case on a compact CR manifold
satisfying the Lichnerowicz type condition was proved in \cite{LW,LW1} using
the results and the method of \cite{IV3} while introducing a new integration
by parts step proving the vanishing of the Webster torsion assuming the
first eigenvalue is equal to $n$ (for the three dimensional case see \cite%
{IV4}). On the other hand, a generalization of the {Obata result in the complete non-compact case} was
achieved in \cite{IV3}, where the standard Sasakian structure on the unit
sphere was characterized through the existence of a non-trivial solution of
a (horizontal) Hessian equation on a complete with respect to the associated Riemannian metric CR manifold with a divergence free Webster
torsion. To the best of our knowledge the case of a general torsion
remains still open.

The main purpose of this paper is to prove the qc version of both
  results  of Obata under no extra assumptions on the Biquard'
torsion {   when the dimension of the qc manifold is at least eleven}, cf. Theorem~\ref{main1} and Theorem~\ref{main2}.

The quaternionic contact version of the Lichnerowicz' result was found in
\cite{IPV1} in dimensions grater than seven and in \cite{IPV2} in the seven
dimensional case. The following result of \cite{IPV1} gives a lower bound on
the positive eigenvalues of the sub-Laplacian on a qc manifold.

\begin{thrm}[\protect\cite{IPV1}]
\label{mainpan} Let $(M,\eta,g,\mathbb{Q})$ be a compact quaternionic
contact manifold of dimension $4n+3>7$. Suppose that there is a positive
constant $k_0$ such that the qc Ricci tensor and torsion of the Biquard
connection satisfy the inequality
\begin{equation}  \label{condm}
Ric(X,X)+\frac{2(4n+5)}{2n+1}T^0(X,X)+ \frac{6(2n^2+5n-1)}{(n-1)(2n+1)}%
U(X,X)\ge k_0g(X,X).
\end{equation}
Then any eigenvalue $\lambda$ of the sub-Laplacian $\triangle$ satisfies the
inequality
\begin{equation*}
\lambda \ge \frac{n}{n+2}k_0
\end{equation*}
\end{thrm}
{The equality case of Theorem~\ref{mainpan} is achieved on the 3-Sasakian sphere.}
It was shown in \cite{IMV2}, see also \cite{ACB}, that the eigenspace of the
first non-zero eigenvalue of the sub-Laplacian on the unit 3-Sasakian sphere
in Euclidean space is given by the restrictions to the sphere of all linear
functions.

The main results of this paper are the following three theorems.

\begin{thrm}
\label{main1} Let $(M,\eta,g,\mathbb{Q})$ be a compact quaternionic contact
manifold of dimension $4n+3>7$ whose qc-Ricci tensor and torsion of the
Biquard connection satisfy the inequality \eqref{condm}. Then, the first
positive eigenvalue $\lambda$ of the sub-Laplacian $\triangle$ satisfies the
equality
\begin{equation}  \label{feing}
\lambda = \frac{n}{n+2}k_0
\end{equation}
if and only if the qc manifold $(M,g,\mathbb{Q})$ is  qc-homothetic to the
unit (4n+3)-dimensional 3-Sasakian sphere.
\end{thrm}

According to \cite[Remark~4.1]{IPV1}, under the conditions of
Theorem~\ref{mainpan}, an eigenfunction $f$ corresponding to the
the first non-zero eigenvalue as in \eqref{feing}, $\triangle
f=\frac{n}{n+2}k_0\,f$  satisfies a linear PDE system, namely, the
horizontal
Hessian of  $f$ is given by (see Corollary~%
\ref{imp} in the Appendix)
\begin{equation}
\nabla df(X,Y)=-\frac{1}{4(n+2)}k_0fg(X,Y)-\sum_{s=1}^{3}df(\xi _{s})\omega
_{s}(X,Y).  \label{eq07}
\end{equation}
This brings us to our second main result.
\begin{thrm}
\label{main2} Let $(M,\eta,g,\mathbb{Q})$ be a quaternionic contact manifold
of dimension $4n+3>7$ which is complete with respect to the associated Riemannian
metric
\begin{equation}  \label{riem1}
h=g+(\eta_1)^2+(\eta_2)^2+(\eta_3)^2.
\end{equation}
Suppose there exists a non-constant smooth function $f$ whose horizontal
Hessian satisfies
\begin{equation}
\nabla df(X,Y)=-fg(X,Y)-\sum_{s=1}^{3}df(\xi _{s})\omega _{s}(X,Y).
\label{eq7}
\end{equation}
Then the qc manifold $(M,\eta,g,\mathbb{Q})$ is qc homothetic to the unit
(4n+3)-dimensional 3-Sasakian sphere.
\end{thrm}
Clearly Theorem~\ref{main2} implies Theorem~\ref{main1} since any
Riemannian metric on a compact manifold is complete and a qc-homothety, cf. Definition \ref{d:3-ctct auto}, allows us
to reduce to the case  $k_0=4(n+2)$, which turns \eqref{eq07} in \eqref{eq7}.

 We achieve the proof of Theorem~\ref{main2} by showing first that $M$ is isometric  to the unit sphere $S^{4n+3}$ and then that $M$ is qc-equivalent to the standard 3-Sasakian structure on $S^{4n+3}$.  To this effect we show that the torsion of the Biquard connection vanishes and in this case the Riemannian Hessian
satisfies \eqref{clasob} after which we invoke the classical Obata theorem showing that $M$ is isometric  to the unit sphere. In order to prove the qc-equivalence part we show that  the qc-conformal curvature vanishes,  which gives
the local qc conformal equivalence with the 3-Sasakian sphere due to \cite[Theorem~1.3]{IV},  and then use the below  Liouville-type result Theorem \ref{l:Liouville}, which implies  the existence of a global qc-conformal  map between $M$ and  the 3-Sasakian sphere, cf. sub-section \ref{ss:qc conf flat}.
\begin{thrm}\label{l:Liouville}
Let $\Sigma \subset {S^{4n+3}}$ be a connected open subset of ${S^{4n+3}}$. If $F:\Sigma\rightarrow {S^{4n+3}}$ is a qc-conformal transformation then $F$ is the restriction to $\Sigma$  of an element of $PSp(n+1,1)$ - the isometry group of the quaternionic hyperbolic space.
\end{thrm}
The proof of Theorem \ref{l:Liouville} and background to this result can be found in sub-sections \ref{ss:qc conf flat} and \ref{ss:Liouville}.

A version of Theorem~\ref{mainpan} when $n=1$ was established in \cite[%
Theorem~1.1]{IPV2} assuming the positivity of the $P$-function of any
eigenfunction. In the Appendix, for completeness, we recall the notion of
the $P$-function introduced in \cite{IPV2} and give a different proof of
Theorem~\ref{mainpan} based on the positivity of the $P$-function in the
case $n>1$ established in \cite[Theorem~3.3]{IPV2}. As a corollary of the
proof, we show the validity of \eqref{eq07} for any eigenfunction of the sub-Laplacian with eigenvalue given by \eqref{feing}.

\begin{conv}
\label{conven} \hfill\break\vspace{-15pt}

\begin{enumerate}[a)]
\item We shall use $X,Y,Z,U$ to denote horizontal vector fields, i.e. $%
X,Y,Z,U\in H$.

\item  $\{e_1,\dots,e_{4n}\}$ denotes a local orthonormal basis of the
horizontal space $H$.

\item The summation convention over repeated vectors from the basis $%
\{e_1,\dots,e_{4n}\}$ will be used. For example, for a (0,4)-tensor $P$, the
formula $k=P(e_b,e_a,e_a,e_b)$ means $k=\sum_{a,b=1}^{4n}P(e_b,e_a,e_a,e_b)$.

\item The triple $(i,j,k)$ denotes any cyclic permutation of $(1,2,3)$.

\item The sum $\sum_{(ijk)}$ means the cyclic sum. For example,
\[
\sum_{(ijk)} df(I_i X)\omega_j (Y,Z)=df(I_1 X)\omega_2 (Y,Z)+df(I_2 X)\omega_3 (Y,Z)+df(I_3 X)\omega_1 (Y,Z).
\]

\item[e)] $s$ will be any number from the set $\{1,2,3\}, \quad
s\in\{1,2,3\} $.
\end{enumerate}
\end{conv}

\textbf{Acknowledgments} The research is partially supported by the Contract
``Idei", DID 02-39/21.12.2009. S.I and A.P. are partially supported by the
Contract 130/2012 with the University of Sofia `St.Kl.Ohridski'. D.V. would like to thank Professor Luca Capogna for some useful comments.

\section{Quaternionic contact manifolds}

In this section we will briefly review the basic notions of quaternionic
contact geometry and recall some results from \cite{Biq1}, \cite{IMV} and
\cite{IV} which we will use in this paper.

It is well known that the sphere at infinity of a  non-compact symmetric space $M$ of rank one carries a natural
Carnot-Carath\'eodory structure, see \cite{M,P}.  In the real hyperbolic case one obtains the conformal class of the round metric on the sphere. In the remaining cases, each of the complex, quaternion and octonionic   hyperbolic metrics on the unit ball induces a Carnot-Carath\'eodory  structure on the unit sphere.  This defines a conformal structure  on a sub-bundle of  the tangent bundle of co-dimension $\dim \mathbb{K} -1$, where $\mathbb{K}=\mathbb{C},\,  \mathbb{H},\,  \mathbb{O}$. In the complex case the obtained geometry is the well studied standard CR structure on the unit sphere in complex space.  Quaternionic contact (qc) structure were introduced by O. Biquard, see \cite{Biq1}, and are modeled on  the conformal boundary at infinity of the quaternionic hyperbolic space. Biquard showed that the infinite dimensional family \cite{LeB91} of complete quaternionic-K\"ahler deformations of the quaternion hyperbolic metric have conformal infinities which provide an infinite dimensional family of examples of qc structures. Conversely, according to \cite{Biq1} every real analytic qc structure on a manifold $M$ of dimension at least eleven is the conformal infinity of a unique quaternionic-K\"ahler metric defined in a neighborhood of $M$. Furthermore,  \cite{Biq1} considered CR and qc structures as boundaries of infinity of Einstein metrics rather than only as boundaries at infinity of  K\"ahler-Einstein and quaternionic-K\"ahler metrics, respectively.  In fact, in \cite{Biq1} it was shown that in each of the three cases (complex, quaternionic, octoninoic)  any small perturbation of the standard Carnot-Carath\'eodory structure on the boundary is the conformal infinity of an essentially unique Einstein metric on the unit ball, which is asymptotically symmetric.  In the Riemannian case the corresponding question was posed  in \cite{FGr85} and  the perturbation result was proven in \cite{GrL91}.

{Another natural extension of an interesting Riemannian problem is the quaternionic contact  Yamabe problem, a particular case of which \cite{GV,Wei,IMV,IMV1}  amounts to finding  the best constant in the $L^2$
Folland-Stein Sobolev-type embedding and the functions for which the equality is achieved, \cite{F2} and \cite{FS} with a complete solution  on
the  quaternionic Heisenberg groups  given in \cite{IMV1,IMV2}.}

\subsection{Quaternionic contact structures and the Biquard connection}
A quaternionic contact (qc) manifold $(M, \eta,g, \mathbb{Q})$ is a $4n+3$%
-dimensional manifold $M$ with a codimension three distribution $H$ locally
given as the kernel of a 1-form $\eta=(\eta_1,\eta_2,\eta_3)$ with values in
$\mathbb{R}^3$. In addition $H$ has an $Sp(n)Sp(1)$ structure, that is, it
is equipped with a Riemannian metric $g$ and a rank-three bundle $\mathbb{Q}$
consisting of endomorphisms of $H$ locally generated by three almost complex
structures $I_1,I_2,I_3$ on $H$ satisfying the identities of the imaginary
unit quaternions, $I_1I_2=-I_2I_1=I_3, \quad I_1I_2I_3=-id_{|_H}$ which are
hermitian compatible with the metric $g(I_s.,I_s.)=g(.,.)$ and the following
compatibility condition holds $\qquad 2g(I_sX,Y)\ =\ d\eta_s(X,Y), \quad
X,Y\in H.$

The transformations preserving a given quaternionic contact
structure $\eta$, i.e., $\bar\eta=\mu\Psi\eta$ for a positive smooth
function $\mu$ and an $SO(3)$ matrix $\Psi$ with smooth functions as
entries are called \emph{quaternionic contact conformal (qc-conformal) transformations}, see Deinition \ref{d:3-ctct auto}.  If the function $\mu$ is constant $\bar\eta$ is called qc-homothetic to $\eta$. The qc conformal curvature tensor $W^{qc}$, introduced in \cite{IV}, is the
obstruction for a qc structure to be locally qc conformal to the
standard 3-Sasakian structure on the $(4n+3)$-dimensional sphere \cite{IMV,IV}.

A special phenomena, noted in \cite{Biq1}, is that the contact form $\eta$
determines the quaternionic structure and the metric on the horizontal
distribution in a unique way.

On a qc manifold with a fixed metric $g$ on $H$ there exists a canonical
connection defined first by O. Biquard in \cite{Biq1} when the dimension $(4n+3)>7$, and in \cite%
{D} for the 7-dimensional case.
Biquard showed that
there  is a unique connection $\nabla$ with torsion $T$  and a
unique supplementary subspace $V$ to $H$ in $TM$, such that:
\begin{enumerate}[(i)]
\item $\nabla$ preserves the decomposition $H\oplus V$ and the $
Sp(n)Sp(1)$ structure on $H$, i.e. $\nabla g=0, \nabla\sigma \in\Gamma(
\mathbb{Q})$ for a section $\sigma\in\Gamma(\mathbb{Q})$, and its torsion on
$H$ is given by $T(X,Y)=-[X,Y]_{|V}$;
\item for $\xi\in V$, the endomorphism $T(\xi,.)_{|H}$ of $H$ lies in $
(sp(n)\oplus sp(1))^{\bot}\subset gl(4n)$;
\item the connection on $V$ is induced by the natural identification $
\varphi$ of $V$ with the subspace $sp(1)$ of the endomorphisms of $H$, i.e. $
\nabla\varphi=0$.
\end{enumerate}
This canonical connection is also known as \emph{the Biquard connection}.When the dimension of $M$ is
at least eleven \cite{Biq1} also described the supplementary distribution $V$%
, which is (locally) generated by the so called Reeb vector fields $%
\{\xi_1,\xi_2,\xi_3\}$ determined by
\begin{equation}  \label{bi1}
\begin{aligned} \eta_s(\xi_k)=\delta_{sk}, \qquad (\xi_s\lrcorner
d\eta_s)_{|H}=0,\\ (\xi_s\lrcorner d\eta_k)_{|H}=-(\xi_k\lrcorner
d\eta_s)_{|H}, \end{aligned}
\end{equation}
where $\lrcorner$ denotes the interior multiplication. If the dimension of $%
M $ is seven Duchemin shows in \cite{D} that if we assume, in addition, the
existence of Reeb vector fields as in \eqref{bi1}, then   the Biquard result  holds. Henceforth, by a qc structure in dimension $7$ we shall mean a qc structure satisfying \eqref{bi1}.

Notice that equations \eqref{bi1} are invariant under the natural $SO(3)$
action. Using the triple of Reeb vector fields we extend the metric $g$ on $%
H $ to a metric $h$ on $M$ by requiring $span\{\xi_1,\xi_2,\xi_3\}=V\perp H
\text{ and } h(\xi_s,\xi_k)=\delta_{sk}. $  The
Riemannian metric $h$ as well as the Biquard connection do not depend on the action of $SO(3)$ on $V$, but both
change  if $\eta$ is multiplied by a conformal factor \cite{IMV}.
Clearly, the Biquard connection preserves the Riemannian metric on $TM, \nabla
h=0$. Since the Biquard connection is metric it is connected with the
Levi-Civita connection $\nabla^h$ of the metric $h$ by the general formula
\begin{equation}  \label{lcbi}
h(\nabla_AB,C)=h(\nabla^h_AB,C)+\frac12\Big[ %
h(T(A,B),C)-h(T(B,C),A)+h(T(C,A),B)\Big], \quad A,B,C\in\Gamma(TM).
\end{equation}

The covariant derivative of the qc structure with respect to the Biquard
connection and the covariant derivative of the distribution $V $ are given
by
\begin{equation}  \label{xider}
\nabla I_i=-\alpha_j\otimes I_k+\alpha_k\otimes I_j,\quad
\nabla\xi_i=-\alpha_j\otimes\xi_k+\alpha_k\otimes\xi_j.
\end{equation}
The vanishing of the $sp(1)$-connection $1$-forms on $H$ implies the
vanishing of the torsion endomorphism of the Biquard connection (see \cite%
{IMV}).

The fundamental 2-forms $\omega_s$ of the quaternionic structure $\mathbb{Q}$ are
defined by
\begin{equation}  \label{thirteen}
2\omega_{s|H}\ =\ \, d\eta_{s|H},\qquad \xi\lrcorner\omega_s=0,\quad \xi\in
V.
\end{equation}
Due to \eqref{thirteen}, the torsion restricted to $H$ has the form
\begin{equation}  \label{torha}
T(X,Y)=-[X,Y]_{|V}=2\omega_1(X,Y)\xi_1+2\omega_2(X,Y)\xi_2+2\omega_3(X,Y)%
\xi_3.
\end{equation}

\subsection{Invariant decompositions}

An endomorphism $\Psi$ of $H$ can be decomposed with respect to the
quaternionic structure $(\mathbb{Q},g)$ uniquely into four $Sp(n)$-invariant
parts 
$\Psi=\Psi^{+++}+\Psi^{+--}+\Psi^{-+-}+\Psi^{--+},$ 
where $\Psi^{+++}$ commutes with all three $I_i$, $\Psi^{+--}$ commutes with
$I_1$ and anti-commutes with the others two and etc. The two $Sp(n)Sp(1)$%
-invariant components $\Psi_{[3]}=\Psi^{+++}, \quad
\Psi_{[-1]}=\Psi^{+--}+\Psi^{-+-}+\Psi^{--+} $ are determined by
\begin{equation*}
\begin{aligned} \Psi=\Psi_{[3]} \quad \Longleftrightarrow 3\Psi+I_1\Psi
I_1+I_2\Psi I_2+I_3\Psi I_3=0,\\ \Psi=\Psi_{[-1]}\quad \Longleftrightarrow
\Psi-I_1\Psi I_1-I_2\Psi I_2-I_3\Psi I_3=0. \end{aligned}
\end{equation*}
With a short calculation one sees that the $Sp(n)Sp(1)$-invariant components
are the projections on the eigenspaces of the Casimir operator $\Upsilon =\
I_1\otimes I_1\ +\ I_2\otimes I_2\ +\ I_3\otimes I_3$ corresponding,
respectively, to the eigenvalues $3$ and $-1$, see \cite{CSal}. If $n=1$
then the space of symmetric endomorphisms commuting with all $I_s$ is
1-dimensional, i.e. the [3]-component of any symmetric endomorphism $\Psi$
on $H$ is proportional to the identity, $\Psi_{3}=-\frac{tr\Psi}{4}Id_{|H}$%
. Note here that each of the three 2-forms $\omega_s$ belongs to its
[-1]-component, $\omega_s=\omega_{s[-1]}$ and constitute a basis of the Lie
algebra $sp(1)$.

\subsection{The torsion tensor}

The properties of the Biquard connection are encoded in the properties of
the torsion endomorphism $T_{\xi }=T(\xi ,\cdot ):H\rightarrow H,\quad \xi
\in V$. Decomposing the endomorphism $T_{\xi }\in (sp(n)+sp(1))^{\perp }$
into its symmetric part $T_{\xi }^{0}$ and skew-symmetric part $b_{\xi
},T_{\xi }=T_{\xi }^{0}+b_{\xi }$, O. Biquard shows in \cite{Biq1} that the
torsion $T_{\xi }$ is completely trace-free, $tr\,T_{\xi }=tr\,T_{\xi }\circ
I_{s}=0$, its symmetric part has the properties $T_{\xi
_{i}}^{0}I_{i}=-I_{i}T_{\xi _{i}}^{0}\quad I_{2}(T_{\xi
_{2}}^{0})^{+--}=I_{1}(T_{\xi _{1}}^{0})^{-+-},\quad I_{3}(T_{\xi
_{3}}^{0})^{-+-}=I_{2}(T_{\xi _{2}}^{0})^{--+},\quad I_{1}(T_{\xi
_{1}}^{0})^{--+}=I_{3}(T_{\xi _{3}}^{0})^{+--}$, where the upperscript $+++$
means commuting with all three $I_{i}$, $+--$ indicates commuting with $%
I_{1} $ and anti-commuting with the other two and etc. The skew-symmetric
part can be represented as $b_{\xi _{i}}=I_{i}u$, where $u$ is a traceless
symmetric (1,1)-tensor on $H$ which commutes with $I_{1},I_{2},I_{3}$.
Therefore we have$T_{\xi _{i}}=T_{\xi _{i}}^{0}+I_{i}u$. If $n=1$ then the
tensor $u$ vanishes identically, $u=0$, and the torsion is a symmetric
tensor, $T_{\xi }=T_{\xi }^{0}$.

Any 3-Sasakian manifold has zero torsion endomorphism, $T_{\xi}=0$, and the converse is
true if in addition the qc scalar curvature (see \eqref{qscs}) is a positive
constant \cite{IMV}  {(the case of negative qc-scalar curvature  can be treated very similarly, see \cite{IV1,IV2}}. We remind that a $(4n+3)$-dimensional Riemannian
manifold $(M,g)$ is called 3-Sasakian if the cone metric $g_c=t^2h+dt^2$ on $%
C=M\times \mathbb{R}^+$ is a hyper K\"ahler metric, namely, it has holonomy contained in $Sp(n+1)$ \cite{BGN}. A 3-Sasakian manifold of dimension $%
(4n+3) $ is Einstein with positive Riemannian scalar curvature $(4n+2)(4n+3)$
\cite{Kas} and if complete it is a compact manifold with a finite fundamental group
(see \cite{BG} for a nice overview of 3-Sasakian spaces).

\subsection{Torsion and curvature}

Let $R=[\nabla,\nabla]-\nabla_{[\ ,\ ]}$ be the curvature tensor of $\nabla$
and the dimension is $4n+3$. We denote the curvature tensor of type (0,4)
and the torsion tensor of type (0,3) by the same letter, $%
R(A,B,C,D):=h(R(A,B)C,D),\quad T(A,B,C):=h(T(A,B),C)$, $A,B,C,D \in
\Gamma(TM)$. The Ricci tensor, the normalized scalar curvature, the Ricci $2$%
-forms and the Ricci type-tensor $\zeta$ of the Biquard connection, called
\emph{qc-Ricci tensor} $Ric$, \emph{normalized qc-scalar curvature} $S$,
\emph{qc-Ricci forms} $\rho_s$, respectively, are given by
\begin{equation}  \label{qscs}
\begin{aligned} Ric(A,B)=R(e_b,A,B,e_b),\quad 8n(n+2)S=R(e_b,e_a,e_a,e_b),\\
\rho_s(A,B)=\frac1{4n}R(A,B,e_a,I_se_a),\quad
\zeta_s(A,B)=\frac1{4n}R(e_a,A,B,I_se_a). \end{aligned}
\end{equation}
The $sp(1)$-part of $R$ is determined by the Ricci 2-forms and the
connection 1-forms by
\begin{equation}  \label{sp1curv}
R(A,B,\xi_i,\xi_j)=2\rho_k(A,B)=(d\alpha_k+\alpha_i\wedge\alpha_j)(A,B),
\qquad A,B \in \Gamma(TM).
\end{equation}
The two $Sp(n)Sp(1)$-invariant trace-free symmetric 2-tensors $T^0(X,Y)=
g((T_{\xi_1}^{0}I_1+T_{\xi_2}^{0}I_2+T_{ \xi_3}^{0}I_3)X,Y)$, $U(X,Y)
=g(uX,Y)$ on $H$, introduced in \cite{IMV}, have the properties:
\begin{equation}  \label{propt}
\begin{aligned} T^0(X,Y)+T^0(I_1X,I_1Y)+T^0(I_2X,I_2Y)+T^0(I_3X,I_3Y)=0, \\
U(X,Y)=U(I_1X,I_1Y)=U(I_2X,I_2Y)=U(I_3X,I_3Y). \end{aligned}
\end{equation}
In dimension seven $(n=1)$, the tensor $U$ vanishes identically, $U=0$.

We shall need the following identity taken from \cite[Proposition~2.3]{IV} $%
4T^0(\xi_s,I_sX,Y)=T^0(X,Y)-T^0(I_sX,I_sY)$ which implies the formula
\begin{equation}  \label{need1}
T(\xi_s,I_sX,Y)=T^0(\xi_s,I_sX,Y)+g(I_suI_sX,Y)=\frac14\Big[%
T^0(X,Y)-T^0(I_sX,I_sY)\Big]-U(X,Y).
\end{equation}
We recall that a qc structure is said to be \emph{qc-Einstein} if the horizontal
qc-Ricci tensor is a scalar multiple of the metric, $Ric(X,Y)=2(n+2)Sg(X,Y)$.
The horizontal Ricci-type tensor can be expressed in terms of the torsion of
the Biquard connection \cite{IMV} (see also \cite{IMV1,IV}). We collect
below the necessary facts from \cite[Theorem~1.3, Theorem~3.12,
Corollary~3.14, Proposition~4.3 and Proposition~4.4]{IMV} with slight
modification presented in \cite{IV}
\begin{equation}  \label{sixtyfour}
\begin{aligned} & Ric(X,Y) =(2n+2)T^0(X,Y)+(4n+10)U(X,Y)+2(n+2)Sg(X,Y),\\ &
\rho_s(X,I_sY) =
-\frac12\Bigl[T^0(X,Y)+T^0(I_sX,I_sY)\Bigr]-2U(X,Y)-Sg(X,Y),\\ &
\zeta_s(X,I_sY)=\frac{2n+1}{4n}T^0(X,Y)+\frac1{4n}T^0(I_sX,I_sY)+%
\frac{2n+1}{2n}U(X,Y)+\frac{S}2g(X,Y), \\ & T(\xi_{i},\xi_{j})
=-S\xi_{k}-[\xi_{i},\xi_{j}]_{H}, \qquad S = -h(T(\xi_1,\xi_2),\xi_3),\\ &
g(T(\xi_i,\xi_j),X)
=-\rho_k(I_iX,\xi_i)=-\rho_k(I_jX,\xi_j)=-h([\xi_i,\xi_j],X). \end{aligned}
\end{equation}
For $n=1$ the above formulas hold with $U=0$.  Hence, the qc-Einstein condition is equivalent to the vanishing of the
torsion endomorphism of the Biquard connection. In this case the normalized qc scalar curvature $S$ is
constant and the vertical distribution $V$ is integrable provided $n>1$. If $S>0$
then the qc manifold is locally 3-Sasakian \cite{IMV}, (see \cite{IV1} for
the negative qc scalar curvature).

We shall also need the general formula for the curvature \cite{IV,IV2}
\begin{multline}
R(\xi _{i},X,Y,Z)=-(\nabla _{X}U)(I_{i}Y,Z)+\omega _{j}(X,Y)\rho
_{k}(I_{i}Z,\xi _{i})-\omega _{k}(X,Y)\rho _{j}(I_{i}Z,\xi _{i})
\label{d3n5} \\
-\frac{1}{4}\Big[(\nabla _{Y}T^{0})(I_{i}Z,X)+(\nabla _{Y}T^{0})(Z,I_{i}X)%
\Big]+\frac{1}{4}\Big[(\nabla _{Z}T^{0})(I_{i}Y,X)+(\nabla
_{Z}T^{0})(Y,I_{i}X)\Big] \\
-\omega _{j}(X,Z)\rho _{k}(I_{i}Y,\xi _{i})+\omega _{k}(X,Z)\rho
_{j}(I_{i}Y,\xi _{i})-\omega _{j}(Y,Z)\rho _{k}(I_{i}X,\xi _{i})+\omega
_{k}(Y,Z)\rho _{j}(I_{i}X,\xi _{i}),
\end{multline}%
where the Ricci two forms are given by, cf. \cite[Theorem 3.1]{IV} or \cite[%
Theorem4.3.11]{IV2}
\begin{equation}
\begin{aligned}
6(2n+1)\rho_s(\xi_s,X)=(2n+1)X(S)+\frac12(%
\nabla_{e_a}T^0)[(e_a,X)-3(I_se_a,I_sX)]-2(\nabla_{e_a}U)(e_a,X),\\
6(2n+1)\rho_i(\xi_j,I_kX)=(2n-1)(2n+1)X(S)-\frac12(%
\nabla_{e_a}T^0)[(4n+1)(e_a,X)+3(I_ie_a,I_iX)]\\-4(n+1)(%
\nabla_{e_a}U)(e_a,X) .\end{aligned}  \label{d3n6}
\end{equation}

\subsection{The Ricci identities, the divergence theorem.}

We shall use repeatedly the following Ricci identities of order two and
three, see also \cite{IV} and\cite{IPV1}. 
Let $\xi_i$, $i=1,2,3$ be the Reeb vector fields,  $f$ a smooth
function on the qc manifold $M$ and $\nabla f$ its horizontal
gradient, $g(\nabla f,X)=df(X)$. We have:
\begin{equation}  \label{boh2}
\begin{aligned} & \nabla^2f
(X,Y)-\nabla^2f(Y,X)=-2\sum_{s=1}^3\omega_s(X,Y)df(\xi_s),\\
& \nabla^2f
(X,\xi_s)-\nabla^2f(\xi_s,X)=T(\xi_s,X,\nabla f),\\ & \nabla^3 f
(X,Y,Z)-\nabla^3 f(Y,X,Z)=-R(X,Y,Z,\nabla f) - 2\sum_{s=1}^3
\omega_s(X,Y)\nabla^2f (\xi_s,Z),\\ &
\nabla^3f(X,Y,\xi_i)-\nabla^3f(Y,X,\xi_i)=-2df(\xi_j)\rho_k(X,Y)+2df(\xi_k)
\rho_j(X,Y)-2\sum_{s=1}^3\omega_s(X,Y)\nabla^2(\xi_s,\xi_i),\\&\nabla^3f(%
\xi_s,X,Y)-\nabla^3f(X,\xi_s,Y)=-R(\xi_s,X,Y,\nabla
f)-\nabla^2f(T(\xi_s,X),Y),\\& \nabla ^{3}f(\xi _{s},X,Y) -\nabla
^{3}f(X,Y,\xi _{s})=-\nabla ^{2}f\left( T\left( \xi _{s},X\right) ,Y\right)
-\nabla ^{2}f\left( X,T\left( \xi _{s},Y\right) \right) -df\left( \left(
\nabla _{X}T\right) \left( \xi _{s},Y\right) \right)\\ &\hskip4.3in -R(\xi
_{s},X,Y,\nabla f). \end{aligned}
\end{equation}

The horizontal sub-Laplacian $\triangle f$ and the norm of the horizontal
gradient $\nabla f$ of a smooth function $f$ on $M$ are defined respectively
by
\begin{equation*}
\triangle f\ =-\ tr^g_H(\nabla^2f)\ =\nabla^*df= -\ \nabla^2f(e_a,e_a),
\qquad |\nabla f|^2\ =\ df(e_a)\,df(e_a).
\end{equation*}
The function $f$ is an eigenfunction with eigenvalue $\lambda$ of the
sub-Laplacian if, for some constant $\lambda$ we have
\begin{equation}  \label{eig}
\triangle f =\lambda f.
\end{equation}
From the Ricci identities we have the following formulas for the traces through the almost complex structures of the Hessian
\begin{equation}  \label{xi1}
g(\nabla^2f , \omega_s) =\nabla^2f(e_a,I_se_a)=-4ndf(\xi_s).
\end{equation}
For a fixed local 1-form $\eta$ and a fix $s\in \{1,2,3\}$ the form $%
Vol_{\eta}=\eta_1\wedge\eta_2\wedge\eta_3\wedge\omega_s^{2n}$ is a locally
defined volume form. Note that $Vol_{\eta}$ is independent on $s $ as well
as it is independent on the local one forms $\eta_1,\eta_2,\eta_3 $. Hence
it is globally defined volume form denoted with $\, Vol_{\eta}$. The
(horizontal) divergence of a horizontal vector field/one-form $%
\sigma\in\Lambda^1\, (H)$ defined by $\nabla^*\, \sigma\
=-tr|_{H}\nabla\sigma=\ -\nabla \sigma(e_a,e_a)$ supplies the "integration
by parts" formula \cite{IMV}, see also \cite{Wei},
\begin{equation}  \label{div}
\int_M (\nabla^*\sigma)\,\, Vol_{\eta}\ =\ 0.
\end{equation}

\section{Proof of the main Theorems}
The proof of Theorem~\protect\ref{main2} is lengthy and requires a number of steps which we present in the following sub-sections. Throughout this section we shall assume the assumptions of Theorem \ref{main2}. In particular, $f$ is a non-constant smooth function whose horizontal Hessian satisfies \eqref{eq7}. Our first step is to show the vanishing of the torsion tensor, $T^0=0$ and $U=0$. We start by expressing the remaining parts of the Hessian (w.r.t. the Biquard connection) in terms of the torsion tensors and show that $f$ satisfies an elliptic equation on $M$.  A simple argument shows that $T^0(I_s\gr,\gr)=U(I_s\gr,\gr)=0$, $s=1,2,3$. Furthermore, using the $[-1]$-component of the curvature tensor we show that $T^0(I_s\gr,I_t\gr)=0$, $s,\, t\in \{1,2,3\}$, $s\not=t$. In addition, we determine the torsion tensors $T^0$ and $U$ in terms of the horizontal gradient of $f$ and the tensor $U(\gr,\gr)$. The analysis proceeds by finding  formulas of the same type for the covariant derivatives of  $T^0$ and $U$.  Thus, the crux of the matter in showing that the torsion vanishes is the proof that $U(\gr,\gr)=0$. This fact will be achieved with the help of the Ricci identities, the contracted Bianchi second identity and thus far established results.   In the next step of the proof of Theorem \ref{main2} we compute the Riemannian Hessian of $f$, with respect to the Levi-Civita connection of the metric \eqref{riem1} which allow us  to invoke Obata's result thus proving that $M$ equipped with the Riemannian metric  \eqref{riem1} is homothetic  to the unit sphere in quaternion space.  The final step is to show that $M$ is qc-homothetic to the $(4n+3)$-dimensional 3-Sasakian unit sphere. Here, we employ a standard monodromy argument showing that a compact simply connected {locally qc-conformally flat manifold is globally qc-conformal to the 3-Sasakian unit sphere.} The key is Theorem~\ref{l:Liouville}, which is a generalization of the Liouville theorem, showing that every qc-conformal transformation between open subsets of the 3-Sasakian unit sphere is the restriction of a global qc-conformal transformation, i.e., an element of the group $PSp(n+1,1)$, see subsection \ref{ss:qc conf flat} for further details.

\subsection{Some basic identities}
We start our analysis by  finding a formula for the the third covariant derivative of a function which satisfies \eqref{eq7}.
\begin{lemma}\label{l:nab3xi}
With the assumptions of Theorem \ref{main2}  we have the following formula for the third covariant derivative of the function $f$,
\begin{equation}\label{nab3xi}
\nabla ^{3}f(A,X,Y)=-df(A)g(X,Y)-\sum_{s=1}^{3}\omega
_{s}(X,Y)\nabla ^{2}f(A,\xi _{s}), \quad A\in\Gamma(TM).
\end{equation}
\end{lemma}

\begin{proof}
The claimed formula is obtained  by differentiating  the Hessian equation \eqref{eq7}.  Indeed, the covariant derivative along $A\in\Gamma(TM)$ of \eqref{eq7} gives
\begin{multline*}
\nabla ^{3}f(A,X,Y)=-df(A)g(X,Y)\\
-\sum_{s=1}^{3}\left[ \nabla ^{2}f(A,\xi _{s})\omega
_{s}(X,Y)+df(\nabla _{A}\xi _{s})\omega _{s}(X,Y)+df(\xi _{s})\left(
\nabla _{A}\omega _{s}\right) (X,Y)\right] ,
\end{multline*}
which together with  \eqref{xider} gives the identity, cf. also Convention 1.4 e),
\begin{multline*}
\nabla ^{3}f(A,X,Y)=-df(A)g(X,Y)-\sum_{(ijk)}\left[ \nabla
^{2}f(A,\xi _{i})\omega _{i}(X,Y)+df(\nabla _{A}\xi
_{i})\omega _{i}(X,Y)+df(\xi _{i})\left( \nabla _{A}\omega
_{i}\right) (X,Y)\right] \\
=-df(A)g(X,Y)-\sum_{t=1}^{3}\left[ \nabla ^{2}f(A,\xi
_{t})\omega _{t}(X,Y)\right]   \\
-\sum_{(ijk)}\left[ -\alpha _{j}(A)df(\xi _{k})+\alpha _{k}(A)df(\xi _{j})\right] \omega _{i}(X,Y)-\sum_{(ijk)}\left[ -\alpha
_{j}(A)\omega _{k}(X,Y)+\alpha _{k}(A)\omega _{j}(X,Y)\right]
df(\xi _{i})  \notag \\
=-df(A)g(X,Y)-\sum_{t=1}^{3}\left[ \nabla ^{2}f(A,\xi
_{t})\omega _{t}(X,Y)\right]  ,
\end{multline*}
which completes the proof.
\end{proof}
After this technical Lemma, our first goal is to find a formula  for  the curvature tensor $R(Z,X,Y,\nabla f)$, for $f$ satisfying \eqref{eq7}, using Lemma \ref{nab3xi} with $A=Z$, the Ricci
identities \eqref{boh2}, and the properties of the torsion.  In fact, after some standard calculations  it follows
\begin{multline}  \label{eqc1}
R(Z,X,Y,\nabla f)=\Big[df(Z)g(X,Y)-df(X)g(Z,Y)\Big] \\
+\sum_{s=1}^3\Big[\nabla df(\xi_s,Z)\omega_s(X,Y)-\nabla
df(\xi_s,X)\omega_s(Z,Y)-2\nabla df(\xi_s,Y)\omega_s(Z,X)\Big] \\
+\sum_{s=1}^3\Big[T(\xi_s,Z,\nabla f)\omega_s(X,Y)-T(\xi_s,X,\nabla
f)\omega_s(Z,Y)\Big].
\end{multline}
By taking traces in \eqref{eqc1}  we can derive formulas for the various contracted tensors \eqref{qscs}. We shall use the following,
\begin{equation}  \label{eqc2}
\begin{aligned} & Ric(Z,\nabla
f)=(4n-1)df(Z)-\sum_{s=1}^3T(\xi_s,I_sZ,\nabla f)-3\sum_{s=1}^3\nabla
df(\xi_s,I_sZ),\\ &4n\zeta_i(I_iZ,\nabla f)=-df(Z)+(4n-1)T(\xi_i,I_iZ,\nabla
f) +T(\xi_j,I_jZ,\nabla f)+T(\xi_k,I_kZ,\nabla f)\\ &\hskip2.5in
+(4n+1)\nabla df(\xi_i,I_iZ)-\nabla df(\xi_j,I_jZ)-\nabla df(\xi_k,I_kZ).
\end{aligned}
\end{equation}
The above formulas imply some other basic identities to which we turn next. Note that with the help of \eqref{sixtyfour} we can rewrite the Lichnerowicz type assumption \eqref{condm}  in the form
\begin{gather}\label{condm-app}
L(X,X)\overset{def}{=}2(n+2) Sg(X,X)+\alpha_n' T^0(X,X) +\beta_n'(X,X)\geq k_0g(X,X), \quad X\in H,\\\nonumber
\alpha_n'=\frac {2(2n+3)(n+2)}{2n+1}, \qquad \beta_n'=\frac {4(2n-1)(n+2)^2}{(2n+1)(n-1)},
\end{gather}
which allows to write the first claim of the following Lemma  in the form $L(Z,\gr)=0$ for all $Z\in H$ whenever $f$ satisfies \eqref{eq7} { taking $k_0=4(n+2)$.}
\begin{lemma}\label{l:Lichn identities}
With the assumptions of Theorem \ref{main2}, the next identity holds true
\begin{equation}  \label{vvvv4}
(S-2)df(Z)+\frac{2n+3}{2n+1}T^0(Z,\nabla f)+\frac{2(2n-1)(n+2)}{(2n+1)(n-1)}%
U(Z,\nabla f)=0.
\end{equation}
Furthermore, we have
\begin{equation}  \label{vvv1}
T^0(I_s\nabla f,\nabla f)=0, \qquad  U(I_s\nabla f,\nabla f)=0.
\end{equation}
\end{lemma}

\begin{proof}
The first equations in \eqref{eqc2} and  \eqref{sixtyfour} together with %
\eqref{need1} imply
\begin{equation}  \label{vvvv}
3\sum_{s=1}^3\nabla df(\xi_s,I_sZ)=\Big[4n-1-(2n+4)S\Big]df(Z)-(2n+3)T^0(Z,%
\nabla f)-(4n+7)U(Z,\nabla f).
\end{equation}
The sum over $1,2,3$ of the second equality in \eqref{eqc2} together with
{the third } equality of \eqref{sixtyfour} and \eqref{need1} gives
\begin{equation}  \label{vvvv1}
(4n-1)\sum_{s=1}^3\nabla df(\xi_s,I_sZ)=(3-6nS)df(Z)-(2n+3)T^0)(Z,\nabla
f)-3U(Z,\nabla f).
\end{equation}
Subtracting \eqref{vvvv} from \eqref{vvvv1} we obtain
\begin{equation*}
4(n-1)\sum_{s=1}^3\nabla df(\xi_s,I_sZ)=4(1-n)(1+S)df(Z)+4(n+1)U(Z,\nabla f),
\end{equation*}
which for $n>1$ yields
\begin{equation}  \label{vvvv2}
\sum_{s=1}^3\nabla df(\xi_s,I_sZ)=-(1+S)df(Z)+\frac{n+1}{n-1}U(Z,\nabla f).
\end{equation}
The sum of \eqref{vvvv} and \eqref{vvvv1} gives
\begin{equation}  \label{vvvv3}
(2n+1)\sum_{s=1}^3\nabla df(\xi_s,I_sZ)=(2n+1)(1-2S)df(Z)-(2n+3)T^0(Z,\nabla
f)-(2n+5)U(Z,\nabla f).
\end{equation}
Equalities \eqref{vvvv2} and \eqref{vvvv3} imply \eqref{vvvv4}.
Letting  $Z=I_s\nabla f$ in the latter it follows $T^0(I_s\nabla f,\nabla f)=0$  since $U(I_s\nabla f,\nabla f)=0$.
\end{proof}

\subsection{Formulas for the  derivatives of $f$}
By assumption, the second order horizontal derivatives of $f$ satisfy the Hessian equation \eqref{eq7}. We derive next   formulas for the second order derivatives involving a horizontal and a vertical directions.
\begin{lemma}\label{l:D2mixed}
With the assumptions of Theorem \ref{main2} we have
\begin{equation}  \label{vvvv5}
\nabla df(\xi_i,I_iZ)=-df(Z)+\frac{2n+3}{4(2n+1)}\Big[T^0(Z,\nabla f)-T^0(I_iZ,I_i\nabla f)\Big]+%
\frac{2n^2+3n-1}{(2n+1)(n-1)}U(Z,\nabla f)
\end{equation}
and
\begin{equation}  \label{vvvv52}
\nabla df(Z,\xi_i)=df(I_iZ)-\frac{n+1}{2n+1}\Big[T^0(I_iZ,\nabla f)+T^0(Z,I_i\nabla f)\Big]-%
\frac{4n}{(2n+1)(n-1)}U(I_iZ,\nabla f).
\end{equation}
\end{lemma}

\begin{proof}
The second equality of \eqref{eqc2} can be written in the form
\begin{multline}  \label{zzz}
4n\zeta_i(I_iZ,\nabla f)=-df(Z)+(4n-2)T(\xi_i,I_iZ,\nabla f)
+\sum_{s=1}^3T(\xi_s,I_sZ,\nabla f) \\
+(4n+2)\nabla df(\xi_i,I_iZ)-\sum_{s=1}^3\nabla df(\xi_s,I_sZ) \\
= -df(Z)+(4n-2)\Big[\frac14(T^0(Z,\nabla f)-T^0(I_iZ,I_i\nabla f))-U(Z,\nabla
f)\Big]+T^0(Z,\nabla f)-3U(Z,\nabla f) \\
+(1+S)df(Z)-\frac{n+1}{n-1}U(Z,\nabla f)+(4n+2)\nabla df(\xi_i,I_iZ),
\end{multline}
where we used \eqref{need1} and \eqref{vvvv2}. Now, equalities \eqref{zzz}, \eqref{vvvv4} and the third equality in
\eqref{sixtyfour} imply
\begin{multline}  \label{e:vvvv5}
\nabla df(\xi_i,I_iZ)=-\frac{S}2df(Z) -\frac{2n+3}{4(2n+1)}\Big[T^0(Z,\nabla
f)+T^0(I_iZ,I_i\nabla f)\Big]+\frac1{(2n+1)(n-1)}U(Z,\nabla f) \\
=-df(Z)+\frac{2n+3}{4(2n+1)}\Big[T^0(Z,\nabla f)-T^0(I_iZ,I_i\nabla f)\Big]+%
\frac{2n^2+3n-1}{(2n+1)(n-1)}U(Z,\nabla f).
\end{multline}
Finally, the Ricci identity, \eqref{need1} and \eqref{vvvv5} yield
\begin{multline}  \label{e:vvvv52}
\nabla^2f(Z,\xi_i)=\nabla df(\xi_i,Z)+T(\xi_i,Z,\nabla f) \\
=\frac{S}2df(I_iZ) +\frac{1}{2(2n+1)}T^0(I_iZ,\nabla f)-\frac{n+1}{2n+1}%
T^0(Z,I_i\nabla f)+\frac{2n^2-n-2}{(2n+1)(n-1)}U(I_iZ,\nabla f) \\
=df(I_iZ)-\frac{n+1}{2n+1}\Big[T^0(I_iZ,\nabla f)+T^0(Z,I_i\nabla f)\Big]-%
\frac{4n}{(2n+1)(n-1)}U(I_iZ,\nabla f),
\end{multline}
which completes the proof.
\end{proof}

{Next, we compute  the second vertical derivatives of $f$. We start with a basic useful  identity involving only vertical derivatives.}
\begin{lemma}\label{l:D2 vert}
With the assumptions of Theorem \ref{main2}  the following identity holds
\begin{multline}  \label{xixi}
\nabla^2f(\xi_i,\xi_i)=-f -\frac{n+1}{4n(2n+1)}\Big[(\nabla_{e_a}T^0)(e_a,\nabla
f)-(\nabla_{e_a}T^0)(I_ie_a,I_i\nabla f)\Big]\\
-\frac{1}{(2n+1)(n-1)}%
(\nabla_{e_a}U)(e_a,\nabla f).
\end{multline}
\end{lemma}

\begin{proof}
Differentiating \eqref{vvvv52},  using \eqref{eq7} {and \eqref{xider} we obtain}
\begin{multline}  \label{bixi111}
\nabla^3f(X,Y,\xi_i)-\alpha_j(X)\bi^2f(Y,\xi_k)+\alpha_k(X)\bi^2f(Y,\xi_j)\\
=-\frac{n+1}{2n+1}\Big[(\nabla_XT^0)(I_iY,\nabla
f)+(\nabla_XT^0)(Y,I_i\nabla f)\Big]-\frac{4n}{(2n+1)(n-1)}(\nabla_XU)(I_iZ,\nabla f) \\
+f\Big\{\omega_i(X,Y)+\frac{n+1}{2n+1}\Big[T^0(X,I_iY)+T^0(I_iX,Y)\Big]+\frac{%
4n}{(2n+1)(n-1)}U(X,I_iY)\Big\} \\
+df(\xi_i)\Big\{-g(X,Y)+\frac{n+1}{2n+1}\Big[T^0(I_iX,I_iY)-T^0(X,Y)\Big]+%
\frac{4n}{(2n+1)(n-1)}U(X,Y)\Big\} \\
+df(\xi_j)\Big\{\omega_k(X,Y)+\frac{n+1}{2n+1}\Big[T^0(I_jX,I_iY)+T^0(I_kX,Y)%
\Big]+\frac{4n}{(2n+1)(n-1)}U(X,I_kY)\Big\} \\
+df(\xi_k)\Big\{-\omega_j(X,Y)+\frac{n+1}{2n+1}\Big[T^0(I_kX,I_iY)-T^0(I_jX,Y)%
\Big]-\frac{4n}{(2n+1)(n-1)}U(X,I_jY)\Big\}\\-\alpha_j(X)\Big[df(I_kY)
-\frac{n+1}{2n+1}T^0(I_kY,\gr)-\frac{n+1}{2n+1}T^0(Y,I_k\gr)-\frac{4n}{(2n+1)(n-1)}U(I_kY\gr)\Big]\\+\alpha_k(X)\Big[df(I_jY)-\frac{n+1}{2n+1}T^0(I_jY,\gr)-\frac{n+1}{2n+1}T^0(Y,I_j\gr)
-\frac{4n}{(2n+1)(n-1)}U(I_jY,\gr)\Big].
\end{multline}
{Applying again \eqref{vvvv52} to the second and the third terms in the first line we see that the terms involving the connection 1-forms cancel and \eqref{bixi111} takes the following form}
\begin{multline}  \label{bixi1}
\nabla^3f(X,Y,\xi_i)\\
=-\frac{n+1}{2n+1}\Big[(\nabla_XT^0)(I_iY,\nabla
f)+(\nabla_XT^0)(Y,I_i\nabla f)\Big]-\frac{4n}{(2n+1)(n-1)}(\nabla_XU)(I_iZ,\nabla f) \\
+f\Big\{\omega_i(X,Y)+\frac{n+1}{2n+1}\Big[T^0(X,I_iY)+T^0(I_iX,Y)\Big]+\frac{%
4n}{(2n+1)(n-1)}U(X,I_iY)\Big\} \\
+df(\xi_i)\Big\{-g(X,Y)+\frac{n+1}{2n+1}\Big[T^0(I_iX,I_iY)-T^0(X,Y)\Big]+%
\frac{4n}{(2n+1)(n-1)}U(X,Y)\Big\} \\
+df(\xi_j)\Big\{\omega_k(X,Y)+\frac{n+1}{2n+1}\Big[T^0(I_jX,I_iY)+T^0(I_kX,Y)%
\Big]+\frac{4n}{(2n+1)(n-1)}U(X,I_kY)\Big\} \\
+df(\xi_k)\Big\{-\omega_j(X,Y)+\frac{n+1}{2n+1}\Big[T^0(I_kX,I_iY)-T^0(I_jX,Y)%
\Big]-\frac{4n}{(2n+1)(n-1)}U(X,I_jY)\Big\}.
\end{multline}
{On the other hand, the skew-symmetric part of \eqref{bixi1} and  the Ricci identity listed in the fourth line of \eqref{boh2}  yield}
\begin{multline}  \label{ricxi}
\nabla^3f(X,Y,\xi_i)-\nabla^3f(Y,X,\xi_i) \\
=-\frac{n+1}{2n+1}\Big[(\nabla_XT^0)(I_iY,\nabla f)+(\nabla_XT^0)(Y,I_i\nabla f)%
-(\nabla_YT^0)(I_iX,\nabla f)-(\nabla_YT^0)(X,I_i\nabla f)\Big]\\
-\frac{4n}{(2n+1)(n-1)}\Big[(\nabla_XU)(I_iY,\nabla
f)-(\nabla_YU)(I_iX,\nabla f)\Big] +2f\Big[\omega_i(X,Y)+\frac{4n}{%
(2n+1)(n-1)}U(X,I_iY)\Big] \\
+2df(\xi_j)\Big\{\omega_k(X,Y)+\frac{n+1}{2n+1}\Big[T^0(I_kX,Y)-T^0(X,I_kY)%
\Big]+\frac{4n}{(2n+1)(n-1)}U(X,I_kY)\Big\} \\
+2df(\xi_k)\Big\{-\omega_j(X,Y)+\frac{n+1}{2n+1}\Big[T^0(X,I_jY)-T^0(I_jX,Y)%
\Big]-\frac{4n}{(2n+1)(n-1)}U(X,I_jY)\Big\}\\
=-2df(\xi_j)\rho_k(X,Y)+2df(\xi_k)\rho_j(X,Y)-2\sum_{s=1}^3\omega_s(X,Y)%
\nabla^2f(\xi_s,\xi_i) .
\end{multline}
The trace  $X=e_a, Y=I_ie_a$ of  \eqref{ricxi}  and the second equality of \eqref{sixtyfour} give \eqref{xixi},
which completes the proof.
\end{proof}
\begin{rmrk}\label{rem00}{ The detailed proof of \eqref{bixi1} shows  a particular consequence of \eqref{xider} which is that  a covariant derivative of identities that are not  $Sp(1)$ invariant can lead to formulas which do not involve the connection one-forms.   In the rest of the paper we shall usually skip many straightforward calculations some of which rely on a similar use of \eqref{xider}.}
\end{rmrk}

\subsection{The elliptic eigenvalue problem}
In this sub-section we  will show that \eqref{eq7} implies that $f$ satisfies an elliptic PDE.
{ Let $\triangle ^h$ be the  Riemannian Laplacian of the  metric \eqref{riem1}. }

\begin{lemma}
\label{rrlem} On a qc manifold of dimension bigger than seven any smooth
function satisfying \eqref{eq7} obeys the following identity
\begin{equation}  \label{llex}
\triangle^h f=(4n+3)f+\frac{n+1}{n(2n+1)}(\nabla_{e_a}T^0)(e_a,\nabla f)+%
\frac{3}{(2n+1)(n-1)}(\nabla_{e_a}U)(e_a,\nabla f).
\end{equation}
\end{lemma}

\begin{proof}
It is shown in \cite[Lemma~5.1]{IPV1} that the Riemannian Laplacian $%
\triangle^h$ and the sub-Laplacian $\triangle$ of a smooth function $f$ are
connected by
\begin{equation}  \label{req18}
\triangle^h f=\triangle f-\sum_{s=1}^3\nabla^2f(\xi_s,\xi_s).
\end{equation}
Equation \eqref{req18} is a consequence of  the formula \eqref{lcbi}, $%
\triangle^hf=-\sum_{a=1}^{4n}\nabla^hdf(e_a,e_a)-\sum_{s=1}^3\nabla^hdf(%
\xi_s,\xi_s)$,  and the identities $T(e_a,A,e_a)=T(\xi_s,A,%
\xi_s)=0$, $A\in \Gamma(TM)$ which follow from the properties of the
torsion tensor $T$ of $\nabla$ listed in \eqref{sixtyfour}.  Lemma \ref{l:D2 vert} and \eqref{propt} imply
\begin{equation}  \label{xixis}
\sum_{s=1}^3\nabla^2f(\xi_s,\xi_s)=-3f-\frac{n+1}{n(2n+1)}%
(\nabla_{e_a}T^0)(e_a,\nabla f)-\frac{3}{(2n+1)(n-1)}(\nabla_{e_a}U)(e_a,%
\nabla f).
\end{equation}
A substitution of \eqref{xixis} in \eqref{req18}, taking into account that $f$ satisfies \eqref{eq7} hence $%
\triangle f=4nf$, we obtain \eqref{llex} which proves the lemma.
\end{proof}
A consequence of Lemma \ref{rrlem} and  Aronsajn's unique continuation result, \cite{Ar}, is that $|\nabla f|$ cannot vanish on any open set. We note this important fact in the next remark.
\begin{rmrk}
\label{eell} If $M$ and $f$ are as in Theorem \ref{main2} then  $|\nabla f|\not=0$ in a dense set since $f\not= const$.
\end{rmrk}

\subsection{Formulas for the torsion tensors}
In this sub-section we derive formulas for  the components $T^0$ and $U$ of the torsion tensor.

\begin{lemma}\label{l:torsions}
With the assumption of Theorem \ref{main2} the following identities hold true for any $X, Y, Z \in H$
\begin{equation}  \label{t02}
T^0(I_s\nabla f,I_t\nabla f)=0, \quad s\not=t, \ s,\, t\in \{1,2,3\},
\end{equation}
\begin{equation}  \label{tu}
T^0(\nabla f,\nabla f)=-\frac{6n}{n-1}U(\nabla f,\nabla f), \qquad
T^0(I_s\nabla f,I_s\nabla f)=\frac{2n}{n-1}U(\nabla f,\nabla f),\qquad  s\in \{1,2,3\},
\end{equation}
\begin{equation}  \label{t0u}
|\nabla f|^2T^0(Z,\nabla f)=-\frac{6n}{n-1}U(\nabla f,\nabla f)df(Z),\qquad
|\nabla f|^2U(Z,\nabla f)=U(\nabla f,\nabla f)df(Z),
\end{equation}
\begin{equation}\label{t000}
|\nabla f|^{4}T^{0}(X,Y)=-\frac{2n}{n-1}U(\nabla f,\nabla f)\Big[%
3df(X)df(Y)-\sum_{s=1}^{3}df(I_{s}X)df(I_{s}Y)\Big],
\end{equation}%
\begin{equation}\label{u000}
|\nabla f|^{4}U(Z,X)=-\frac{1}{n-1}U(\nabla f,\nabla f)\Big[|\nabla
f|^{2}g(Z,X)-n\Big(df(Z)df(X)+\sum_{s=1}^{3}df(I_{s}Z)df(I_{s}X)\Big)\Big].
\end{equation}%
\end{lemma}

\begin{proof}
To determine the torsion tensors $T^0$ and $U$   we are going to apply the
following identity \cite{IV,IV2} for the $[-1]$ component of the curvature
\begin{multline}
3R(Z,X,Y,\nabla f)-R(I_{1}Z,I_{1}X,Y,\nabla f)-R(I_{2}Z,I_{2}X,Y,\nabla
f)-R(I_{3}Z,I_{3}X,Y,\nabla f)  \label{comp1} \\
=2\Big[g(X,Y)T^{0}(Z,\nabla f)+g(Z,\nabla f)T^{0}(Y,X)-g(Y,Z)T^{0}(X,\nabla
f)-g(\nabla f,X)T^{0}(Y,Z)\Big] \\
-2\sum_{s=1}^{3}\Big[\omega _{s}(X,Y)T^{0}(Z,I_{s}\nabla f)+\omega
_{s}(Z,\nabla f)T^{0}(X,I_{s}Y)-\omega _{s}(Z,Y)T^{0}(X,I_{s}\nabla
f)-\omega _{s}(X,\nabla f)T^{0}(Z,I_{s}Y)\Big] \\
+\sum_{s=1}^{3}\Big[2\omega _{s}(Z,X)\Big(T^{0}(Y,I_{s}\nabla
f)-T^{0}(I_{s}Y,\nabla f)\Big)-8\omega _{s}(Y,\nabla f)U(I_{s}Z,X)-4S\omega
_{s}(Z,X)\omega _{s}(Y,\nabla f)\Big].
\end{multline}
With the help of the Ricci identity, cf. the  second equality of %
\eqref{boh2}, we write the curvature tensor given by \eqref{eqc1} in the form
\begin{multline}  \label{vvvv53}
R(Z,X,Y,\nabla f)=\Big[df(Z)g(X,Y)-df(X)g(Z,Y)\Big] \\
+\sum_{s=1}^{3}\Big[\nabla df(Z,\xi _{s})\omega _{s}(X,Y)-\nabla df(X,\xi
_{s})\omega _{s}(Z,Y)-2\nabla df(\xi _{s},Y)\omega _{s}(Z,X)\Big].
\end{multline}
A calculation shows
\begin{multline}  \label{vvvv56}
\sum_{t=1}^3R(I_tZ,I_tX,Y,\nabla f)=\sum_{s=1}^3\Big[df(I_sZ)%
\omega_s(X,Y)-df(I_sX)\omega_s(Z,Y)\Big] \\
+\sum_{s,t=1}^{3}\Big[\nabla df(I_tZ,\xi _{s})\omega _{s}(I_tX,Y)-\nabla
df(I_tX,\xi _{s})\omega _{s}(I_tZ,Y)-2\nabla df(\xi _{s},Y)\omega
_{s}(I_tZ,I_tX)\Big] \\
=\sum_{s=1}^3\Big[df(I_sZ)\omega_s(X,Y)-df(I_sX)\omega_s(Z,Y)+2\nabla df(\xi
_{s},Y)\omega _{s}(Z,X)\Big] \\
-g(X,Y)\sum_{s=1}^{3}\nabla df(I_sZ,\xi _{s})+g(Z,Y)\sum_{s=1}^{3}\nabla df(I_sX,\xi _{s}) -\sum_{(ijk)}\omega_i(X,Y)\Big[%
\nabla df(I_jZ,\xi _{k})-\nabla df(I_kZ,\xi _{j})\Big] \\
 +\sum_{(ijk)}\omega_i(Z,Y)\Big[%
\nabla df(I_jX,\xi _{k})-\nabla df(I_kX,\xi _{j})\Big],
\end{multline}
where $\sum_{(ijk)}$ denotes the cyclic sum. Now, \eqref{vvvv53} and \eqref{vvvv56} together with
\eqref{vvvv5} and \eqref{vvvv52} yield
\begin{multline}  \label{vvvv57}
3R(Z,X,Y,\nabla f)-R(I_{1}Z,I_{1}X,Y,\nabla f)-R(I_{2}Z,I_{2}X,Y,\nabla
f)-R(I_{3}Z,I_{3}X,Y,\nabla f) \\
= g(X,Y)\Big[3df(Z)+\sum_{s=1}^3\nabla^2f(I_sZ,\xi_s)\Big]-g(Z,Y)\Big[%
3df(X)+\sum_{s=1}^3\nabla^2f(I_sX,\xi_s)\Big]-8\sum_{s=1}^3\omega
_{s}(Z,X)\nabla df(\xi _{s},Y) \\
+\sum_{(ijk)}\omega_i(X,Y)\Big[3\nabla^2f(Z,\xi_i)-df(I_iZ)+\nabla^2f(I_jZ,%
\xi_k)-\nabla^2f(I_kZ,\xi_j) \Big] \\
-\sum_{(ijk)}\omega_i(Z,Y)\Big[3\nabla^2f(X,\xi_i)-df(I_iX)+\nabla^2f(I_jX,%
\xi_k)-\nabla^2f(I_kX,\xi_j) \Big] \\
=g(X,Y)\Big[\frac{4n+4}{2n+1}T^0(Z,\nabla f)+\frac{12n}{(2n+1)(n-1)}%
U(Z,\nabla f)\Big] \\
-g(Z,Y)\Big[\frac{4n+4}{2n+1}T^0(X,\nabla f)+\frac{12n}{(2n+1)(n-1)}%
U(X,\nabla f)\Big] \\
-\sum_{s=1}^3\omega _{s}(Z,X)\Big[4Sdf(I_sY) +\frac{4n+6}{2n+1}\Big[%
T^0(I_sY,\nabla f)-T^0(Y,I_s\nabla f)\Big]-\frac{8}{(2n+1)(n-1)}%
U(I_sY,\nabla f) \\
-\sum_{s=1}^3\omega_s(X,Y)\Big\{\frac{4n+4}{2n+1}T^0(Z,I_s\nabla f)+\frac{4n%
}{(2n+1)(n-1)}U(I_sZ,\nabla f)\Big\} \\
+\sum_{s=1}^3\omega_s(Z,Y)\Big\{\frac{4n+4}{2n+1}T^0(X,I_s\nabla f)+\frac{4n%
}{(2n+1)(n-1)}U(I_sX,\nabla f)\Big\}.
\end{multline}
Subtracting \eqref{comp1} from \eqref{vvvv57} and applying \eqref{vvvv5}, %
\eqref{vvvv52} and the properties of the torsion we come to
\begin{multline}  \label{vvvv58}
0= g(X,Y)\Big[T^0(Z,\nabla f)+\frac{6n}{n-1}U(Z,\nabla f)\Big]-g(Z,Y)\Big[%
T^0(X,\nabla f)+\frac{6n}{n-1}U(X,\nabla f)\Big] \\
-\sum_{s=1}^3\omega_s(X,Y)\Big[T^0(Z,I_s\nabla f)+\frac{2n}{n-1}%
U(I_sZ,\nabla f) \Big] +\sum_{s=1}^3\omega_s(Z,Y)\Big[T^0(X,I_s\nabla f)+%
\frac{2n}{n-1}U(I_sX,\nabla f) \Big] \\
-\sum_{s=1}^3\omega_s(Z,X)\Big[2T^0(I_sY,\nabla f)-2T^0(Y,I_s\nabla f)-\frac{%
4}{n-1}U(I_sY,\nabla f) \Big] \\
-(2n+1)\sum_{s=1}^3\Big[%
df(I_sX)T^0(Z,I_sY)-df(I_sZ)T^0(X,I_sY)-4df(I_sY)U(I_sZ,X)\Big] \\
+(2n+1)df(X)T^0(Z,Y)-(2n+1)df(Z)T^0(X,Y).
\end{multline}
Setting $Z=\nabla f$ into \eqref{vvvv58}, after some calculations, we obtain
\begin{multline}  \label{vvvv59}
(2n+1)|\nabla f|^2T^0(X,Y)=(2n+1)df(X)T^0(\nabla f,Y) \\
+g(X,Y)\Big[T^0(\nabla f,\nabla f)+\frac{6n}{n-1}U(\nabla f,\nabla f)\Big]%
-df(Y)\Big[T^0(X,\nabla f)+\frac{6n}{n-1}U(X,\nabla f)\Big] \\
-\sum_{s+1}^3df(I_sY)\Big[T^0(X,I_s\nabla f)+\frac{8n^2-2n-4}{n-1}%
U(I_sX,\nabla f) \Big] \\
-\sum_{s+1}^3df(I_sX)\Big[2T^0(Y,I_s\nabla f)+(2n-1)T^0(I_sY,\nabla f)+\frac{%
4}{n-1}U(I_sY,\nabla f) \Big].
\end{multline}
Letting $Y=\nabla f$ in \eqref{vvvv59}, then using \eqref{t0} and \eqref{vvv1} shows
\begin{equation}  \label{t0}
|\nabla f|^2T^0(X,\nabla f)=T^0(\nabla f,\nabla f)df(X)+\frac{3n}{(n+1)(n-1)}%
\Big[U(\nabla f,\nabla f)df(X)-|\nabla f|^2U(X,\nabla f)\Big].
\end{equation}
On the other hand, letting $X=I_1\nabla f$ in \eqref{vvvv59}, using \eqref{vvv1} and \eqref{t0} gives
\begin{multline}  \label{t01}
0=-df(I_1Y)\Big[ T^0(\nabla f,\nabla f)+T^0(I_1\nabla f,I_1\nabla f) -\frac{%
8n^2-8n-4}{n-1}U(\nabla f,\nabla f)\Big] \\
-df(I_2Y)T^0(I_1\nabla f,I_2\nabla f)-df(I_3Y)T^0(I_1\nabla f,I_3\nabla f) \\
-(2n-1)|\nabla f|^2\Big[T^0(Y,I_1\nabla f)-T^0(I_1Y,\nabla f)\Big]+|\nabla
f|^2\frac{4}{n-1}U(I_1Y,\nabla f).
\end{multline}
From  \eqref{t01} with $Y=I_2\nabla f$  and  \eqref{vvvv1}
the identity \eqref{t02} follows since $|\nabla f|^2\not=0$.

Setting $Y=I_1\nabla f$ into \eqref{t01} implies
\begin{equation}  \label{t012}
T^0(\nabla f,\nabla f)+T^0(I_1\nabla f,I_1\nabla f)=-\frac{4n}{n-1}U(\nabla
f,\nabla f).
\end{equation}
The latter equality together with the \eqref{propt} yield \eqref{tu}.

The equalities \eqref{t01}, \eqref{t02} and \eqref{tu} imply
\begin{multline}  \label{t011}
(2n-1)|\nabla f|^2T^0(Y,I_s\nabla f)=(2n-1)|\nabla f|^2T^0(I_sY,\nabla
f)+|\nabla f|^2\frac{4}{n-1}U(I_sY,\nabla f) \\
-df(I_sY)\Big[ T^0(\nabla f,\nabla f)+T^0(I_s\nabla f,I_s\nabla f) -\frac{%
8n^2-8n-4}{n-1}U(\nabla f,\nabla f)\Big] \\
=(2n-1)|\nabla f|^2T^0(I_sY,\nabla f)+|\nabla f|^2\frac{4}{n-1}U(I_sY,\nabla
f)+4(2n+1)df(I_sY)U(\nabla f,\nabla f).
\end{multline}
Let $Y=I_1\nabla f$ in \eqref{vvvv58} in order to see
\begin{multline}  \label{u}
2(2n+1)|\nabla f|^2U(I_1Z,X)=\omega_1(Z,X)\Big[T^0(\nabla f,\nabla
f)+T^0(I_1\nabla f,I_1\nabla f)-\frac2{n-1}U(\nabla f,\nabla f)\Big] \\
+ndf(X)\Big[T^0(Z,I_1\nabla f)-\frac1{n-1}U(I_1Z,\nabla f)\Big]-ndf(Z)\Big[%
T^0(X,I_1\nabla f)-\frac1{n-1}U(I_1X,\nabla f)\Big] \\
+ndf(I_1X)\Big[T^0(Z,\nabla f)-\frac3{n-1}U(Z,\nabla f)\Big]-ndf(I_1Z)\Big[%
T^0(X,\nabla f)-\frac3{n-1}U(X,\nabla f)\Big] \\
+ndf(I_2X)\Big[T^0(Z,I_3\nabla f)-\frac1{n-1}U(I_3Z,\nabla f)\Big]-ndf(I_2Z)%
\Big[T^0(X,I_3\nabla f)-\frac1{n-1}U(I_3X,\nabla f)\Big] \\
-ndf(I_3X)\Big[T^0(Z,I_2\nabla f)-\frac1{n-1}U(I_2Z,\nabla f)\Big]+ndf(I_3Z)%
\Big[T^0(X,I_2\nabla f)-\frac1{n-1}U(I_2X,\nabla f)\Big].
\end{multline}
Letting $X=\nabla f$ in \eqref{u} and applying \eqref{tu} we obtain
\begin{equation}  \label{tu1}
(4n^2-n-2)U(Z,\nabla f)|\nabla f|^2=(6n^2-n-2)U(\nabla f,\nabla
f)df(Z)-n(n-1)T^0(I_1Z,I_1\nabla f)|\nabla f|^2.
\end{equation}
Therefore,
\begin{equation}  \label{tu2}
|\nabla f|^2\Big[n(n-1)T^0(Z,\nabla f)-3(4n^2-n-2)U(Z,\nabla f)\Big]%
=-3(6n^2-n-2)U(\nabla f,\nabla f)df(Z).
\end{equation}
On the other hand, taking into account \eqref{tu},  equality \eqref{t0} yields
\begin{equation}  \label{tu3}
|\nabla f|^2\Big[(n^2-1)T^0(Z,\nabla f)+3nU(Z,\nabla f)\Big]%
=-3n(2n+1)U(\nabla f,\nabla f)df(Z).
\end{equation}
Solving the system of equations \eqref{tu2} and \eqref{tu3}, we obtain \eqref{t0u}.

A substitution of \eqref{t0u} and \eqref{t0us} in \eqref{vvvv59} gives \eqref{t000}. Now, a substitution of  \eqref{t0u} and \eqref{t0us} in \eqref{u} shows \eqref{u000}.
\end{proof}
We finish this section with a few useful facts. As a direct corollary  from \eqref{vvvv4}, \eqref{t0u} and \eqref{t0us} it follows
\begin{equation} \label{stu}
|\nabla f|^{2}(S-2)=\frac{4(n+1)}{n-1}U(\nabla f,\nabla f).
\end{equation}%
In addition, from  \eqref{t0u} and \eqref{t011} it follows
\begin{equation}  \label{t0us}
|\nabla f|^2T^0(Z,I_s\nabla f)=-\frac{2n}{n-1}U(\nabla f,\nabla f)df(I_sZ),
\quad |\nabla f|^2T^0(I_sY,I_s\nabla f)=\frac{2n}{n-1}U(\nabla f,\nabla
f)df(Z).
\end{equation}
The equalities \eqref{t0u} and \eqref{t0us} yield
\begin{equation}  \label{tss}
T^0(I_sZ,\nabla f)=3T^0(Z,I_s\nabla f), \quad T^0(Z,\nabla
f)=-3T^0(I_sZ,I_s\nabla f),\quad T^0(Z,\nabla f)=-\frac{6n}{n-1}U(Z,\nabla
f),
\end{equation}

\subsection{Formulas for the covariant derivatives of the torsion tensors}
Here we shall prove formulas for the covariant derivative of the torsion tensor.
\begin{lemma}\label{l:div torsion}
If  $M$ and $f$ are as in Theorem \ref{main2}, then we have the following identities for the covariant derivatives of the torsion tensor at the points where $|\gr|\not=0$,
\begin{multline}  \label{nablat0a}
|\nabla f|^2(\nabla_ZT^0)(X,Y)= \frac{4n+2}{n+2}fdf(Z)T^0(X,Y) \\
-\frac{2n}{n-1}\frac{U(\nabla f,\nabla f)}{|\nabla f|^2} f\Big[-3df(Y)g(X,Z)-3df(X)g(Y,Z)
+\sum_{s=1}^3 \left ( df(I_sY)\omega_s(X,Z)+df(I_sX)\omega_s(Y,Z)\right )\Big] \\
-\frac{2n}{n-1}\frac{U(\nabla f,\nabla f)}{|\nabla f|^2} \sum_{(ijk)}df(\xi_i)\Big[%
3df(Y)\omega_i(X,Z)+df(I_iY)g(X,Z)-df(I_jY)\omega_k(X,Z)+df(I_kY)%
\omega_j(X,Z)\Big]\\
-\frac{2n}{n-1}\frac{U(\nabla f,\nabla f)}{|\nabla f|^2} \sum_{(ijk)}df(\xi_i)\Big[%
3df(X)\omega_i(Y,Z)+df(I_iX)g(Y,Z)-df(I_jX)\omega_k(Y,Z)+df(I_kX)%
\omega_j(Y,Z)\Big]
\end{multline}
and
\begin{multline}
|\nabla f|^{2}(\nabla _{Z}U)(X,Y)-2fdf(Z)U(X,Y)-2\sum_{s=1}^{3}df(\xi
_{s})df(I_{s}Z)U(X,Y)=\frac{2n-2}{n+2}fdf(Z)U(X,Y)  \label{vajno} \\
-2\sum_{s=1}^{3}df(\xi _{s})df(I_{s}Z)U(X,Y)-\frac{1}{n-1}\frac{U(\nabla
f,\nabla f)}{|\nabla f|^{2}}\Big[-2fdf(Z)-2\sum_{s=1}^{3}df(\xi
_{s})df(I_{s}Z)\Big]g(X,Y) \\
-\frac{n}{n-1}\frac{U(\nabla f,\nabla f)}{|\nabla f|^{2}}f\Big[%
df(Y)g(X,Z)+ df(X)g(Y,Z)+\sum_{s=1}^3 \left (df(I_{s}Y)\omega _{s}(X,Z)+df(I_{s}X)\omega _{s}(Y,Z)\right )\Big] \\
+\frac{n}{n-1}\frac{U(\nabla f,\nabla f)}{|\nabla f|^{2}}\sum_{(ijk)}df(\xi _{i})\Big[%
df(Y)\omega _{i}(X,Z)-df(I_{i}Y)g(X,Z)+df(I_{j}Y)\omega
_{k}(X,Z)-df(I_{k}Y)\omega _{j}(X,Z)\Big] \\
+\frac{n}{n-1}\frac{U(\nabla f,\nabla f)}{|\nabla f|^{2}}\sum_{(ijk)}df(\xi _{i})\Big[%
df(X)\omega _{i}(Y,Z)-df(I_{i}X)g(Y,Z)+df(I_{j}X)\omega
_{k}(Y,Z)-df(I_{k}X)\omega _{j}(Y,Z)\Big],
\end{multline}
where $\sum_{(ijk)}$ means the cyclic sum, cf. Convention \ref{conven}.
\end{lemma}

\begin{proof}
The contracted Bianchi identity reads \cite{IMV,IV2}
\begin{equation}
(\nabla _{e_{a}}T^{0})(e_{a},X)+\frac{2n+4}{n-1}(\nabla
_{e_{a}}U)(e_{a},X)-(2n+1)dS(X)=0.  \label{conbi2}
\end{equation}%
After taking the trace in the covariant derivatives  of  \eqref{t0u} and \eqref{stu} we obtain
\begin{equation}
\begin{aligned} (\nabla_{e_a}T^0)(e_a,\nabla f)&=-\frac{6n}{n-1}\nabla
f\Big(\frac{U(\gr,\gr)}{|\gr|^2}\Big)+\frac{24n^2}{n-1}f\frac{U(\gr,\gr)}{|%
\gr|^2},\\
(\nabla_{e_a}U)(e_a,\nabla f)&=\nabla
f\Big(\frac{U(\gr,\gr)}{|\gr|^2}\Big)-4nf\frac{U(\gr,\gr)}{|\gr|^2},\\
\nabla f(S)&=\frac{4n+4}{n-1}\nabla f\Big(\frac{U(\gr,\gr)}{|\gr|^2}\Big).
\end{aligned}  \label{nablat0}
\end{equation}%
The system \eqref{nablat0} and \eqref{conbi2} imply
\begin{equation}
\nabla f\Big(\frac{U(\nabla f,\nabla f)}{|\nabla f|^{2}}\Big)=2\frac{n-1}{n+2%
}f\frac{U(\nabla f,\nabla f)}{|\nabla f|^{2}}.  \label{fu}
\end{equation}%
Similarly, using in addition \eqref{tss}, we have
\begin{equation}\label{nablat0s}
\begin{aligned} (\nabla_{e_a}T^0)(e_a,I_s\nabla f)& =\frac{2n}{n-1}I_s\nabla
f\Big(\frac{U(\gr,\gr)}{|\gr|^2}\Big)+\frac{8n^2}{n-1}df(\xi_s)\frac{U(\gr,%
\gr)}{|\gr|^2},\\
 (\nabla_{e_a}U)(e_a,I_s\nabla f)& =I_s\nabla
f\Big(\frac{U(\gr,\gr)}{|\gr|^2}\Big)+4ndf(\xi_s)\frac{U(\gr,\gr)}{|\gr|^2},
\\
I_s\nabla f(S)& = \frac{4n+4}{n-1}I_s\nabla
f\Big(\frac{U(\gr,\gr)}{|\gr|^2}\Big). \end{aligned}
\end{equation}%
Since the differentiation of \eqref{t0us} involves covariant derivatives of the almost
complex structures the derivation of \eqref{nablat0s} requires some care we do it explicitly {again, cf. Remark~\ref{rem00}}. We start with the proof of the first formula in \eqref{nablat0s}.
Differentiating the first equation in \eqref{t0us}, taking into account \eqref{xider}, we have%
\begin{eqnarray*}
&&\left( \nabla _{X}T^{0}\right) (Z,I_{i}\nabla f)-\alpha
_{j}(X)T^{0}(Z,I_{k}\nabla f)+\alpha _{k}(X)T^{0}(Z,I_{j}\nabla
f)+T^{0}(Z,I_{i}\nabla _{X}\left( \nabla f\right) ) \\
&=&-\frac{2n}{n-1}\left[ X\left( \frac{U(\nabla f,\nabla f)}{|\nabla f|^{2}}%
\right) df(I_{i}Z)+\frac{U(\nabla f,\nabla f)}{|\nabla f|^{2}}g\left( \nabla
_{X}\left( \nabla f\right) ,I_{i}Z\right) \right]  \\
&&-\frac{2n}{n-1}\frac{U(\nabla f,\nabla f)}{|\nabla f|^{2}}\left[ -\alpha
_{j}(X)df\left( I_{k}Z\right) +\alpha _{k}(X)df\left( I_{j}Z\right) \right].
\end{eqnarray*}
The  formula for the Hessian \eqref{eq7} gives
\begin{eqnarray*}
&&\left( \nabla _{X}T^{0}\right) (Z,I_{i}\nabla f)-\alpha
_{j}(X)T^{0}(Z,I_{k}\nabla f)+\alpha _{k}(X)T^{0}(Z,I_{j}\nabla
f)-fT^{0}(I_{i}X,Z)-\sum_{s=1}^{3}df\left( \xi _{s}\right)
T^{0}(I_{i}I_{s}X,Z) \\
&=&-\frac{2n}{n-1}\left[ X\left( \frac{U(\nabla f,\nabla f)}{|\nabla f|^{2}}%
\right) df(I_{i}Z)-\frac{U(\nabla f,\nabla f)}{|\nabla f|^{2}}\left(
fg\left( X,I_{i}Z\right) +\sum_{s=1}^{3}df\left( \xi _{s}\right) g\left(
I_{s}X,I_{i}Z\right) \right) \right]  \\
&&-\frac{2n}{n-1}\frac{U(\nabla f,\nabla f)}{|\nabla f|^{2}}\left[ -\alpha
_{j}(X)df\left( I_{k}Z\right) +\alpha _{k}(X)df\left( I_{j}Z\right) \right].
\end{eqnarray*}%
{Taking the trace in the above identity and then applying  the first equation in \eqref{t0us} to the obtained equality we see that the terms involving the connection 1-forms cancel, which gives the first identity in \eqref{nablat0s}.}

The second line in \eqref{nablat0s} follows similarly.

The system \eqref{nablat0s} and \eqref{conbi2} yields
\begin{equation}
I_{s}\nabla f\Big(\frac{U(\nabla f,\nabla f)}{|\nabla f|^{2}}\Big)=2df(\xi
_{s})\frac{U(\nabla f,\nabla f)}{|\nabla f|^{2}}.  \label{fus}
\end{equation}%
We calculate the divergence of $T^{0}$ differentiating \eqref{t000}, taking
the trace in the obtained equality and applying \eqref{fu}, \eqref{fus}.  After a short  computation we obtain
\begin{multline}
|\nabla f|^{2}(\nabla _{e_{a}}T^{0})(e_{a},Y)-2fT^{0}(Y,\nabla
f)+2\sum_{s=1}^{3}df(\xi _{s})T^{0}(Y,I_{s}\nabla f)=  \label{divt0} \\
-\frac{2n}{n-1}\left[ 3\nabla f\left( \frac{U(\nabla f,\nabla f)}{|\nabla
f|^{2}}\right) df(Y)+\sum_{s=1}^{3}I_{s}\nabla f\left( \frac{U(\nabla
f,\nabla f)}{|\nabla f|^{2}}\right) df(I_{s}Y)\right] \\
-\frac{2n}{n-1}\frac{U(\nabla f,\nabla f)}{|\nabla f|^{2}}\left[ 3\nabla
^{2}f(e_{a},e_{a})df(Y)-\sum_{s=1}^{3}\nabla
^{2}f(e_{a},I_{s}e_{a})df(I_{s}Y)\right] \\
-\frac{2n}{n-1}\frac{U(\nabla f,\nabla f)}{|\nabla f|^{2}}\left[ 3\nabla
^{2}f(\nabla f,Y)+\sum_{s=1}^{3}\nabla ^{2}f(I_{s}\nabla f,I_{s}Y)\right] .
\end{multline}%
Applying \eqref{eq7}, \eqref{t0u}, \eqref{t0us}, \eqref{fu} and \eqref{fus} to \eqref{divt0}, we get
\begin{multline}  \label{torto}
|\nabla f|^{2}(\nabla _{e_{a}}T^{0})(e_{a},Y)= \\
-\frac{12n}{n-1}f\frac{U(\nabla f,\nabla f)}{|\nabla f|^{2}}df(Y)+\frac{4n}{%
n-1}\frac{U(\nabla f,\nabla f)}{|\nabla f|^{2}}\sum_{s=1}^{3}df(\xi
_{s})df(I_{s}Y)-\frac{12n}{n+2}\frac{U(\nabla f,\nabla f)}{|\nabla f|^{2}}%
fdf(Y) \\
-\frac{4n}{n-1}\frac{U(\nabla f,\nabla f)}{|\nabla f|^{2}}%
\sum_{s=1}^{3}df(\xi _{s})df(I_{s}Y)+\frac{24n^{2}}{n-1}\frac{U(\nabla
f,\nabla f)}{|\nabla f|^{2}}fdf(Y)-\frac{8n^{2}}{n-1}\frac{U(\nabla f,\nabla
f)}{|\nabla f|^{2}}\sum_{s=1}^{3}df(\xi _{s})df(I_{s}Y) \\
+\frac{6n}{n-1}\frac{U(\nabla f,\nabla f)}{|\nabla f|^{2}}fdf(Y)-\frac{6n}{%
n-1}\frac{U(\nabla f,\nabla f)}{|\nabla f|^{2}}\sum_{s=1}^{3}df(\xi
_{s})df(I_{s}Y) \\
-\frac{2n}{n-1}\frac{U(\nabla f,\nabla f)}{|\nabla f|^{2}}\left[
\sum_{i=1}^{3}\left[ -fdf(Y)+df(\xi _{i})df(I_{i}Y)-df(\xi
_{j})df(I_{j}Y)-df(\xi _{k})df(I_{k}Y)\right] \right] \\
=\frac{12n\left( n+1\right) \left( 2n+1\right) }{\left( n+2\right) \left(
n-1\right) }\frac{U(\nabla f,\nabla f)}{|\nabla f|^{2}}fdf(Y)-4n\frac{2n+1}{%
n-1}\frac{U(\nabla f,\nabla f)}{|\nabla f|^{2}}\sum_{s=1}^{3}df(\xi
_{s})df(I_{s}Y).
\end{multline}%
Applying \eqref{t0u} and \eqref{t0us} to \eqref{torto}, we derive
\begin{equation}
|\nabla f|^{2}(\nabla _{e_{a}}T^{0})(e_{a},Y)=-\frac{(4n+2)(n+1)}{n+2}%
fT^{0}(Y,\nabla f)+(4n+2)\sum_{s=1}^{3}df(\xi _{s})T^{0}(Y,I_{s}\nabla f).
\label{divt00}
\end{equation}%
Now, we calculate the divergence of $U$ differentiating \eqref{u000}, taking
the trace in the obtained equality and applying \eqref{eq7}, \eqref{fu} and \eqref{fus}. We
have
\begin{multline}  \label{divu0}
|\nabla f|^2(\nabla_{e_a}U)(e_a,Y)-2fU(Y,\nabla
f)+2\sum_{s=1}^3df(\xi_s)U(Y,I_s\nabla f)= \\
-\frac1{n-1}Y\Big(\frac{U(\nabla f,\nabla f)}{|\nabla f|^2}\Big)|\nabla f|^2+%
\frac{n}{n-1}\Big[\nabla f\Big(\frac{U(\nabla f,\nabla f)}{|\nabla f|^2}\Big)%
df(Y)-\sum_{s=1}^3I_s\nabla f\Big(\frac{U(\nabla f,\nabla f)}{|\nabla f|^2}%
\Big)df(I_sY)\Big] \\
-\frac{1}{n-1}\frac{U(\nabla f,\nabla f)}{|\nabla f|^2}\Big[%
2\nabla^2f(Y,\nabla
f)-n\nabla^2f(e_a,e_a)df(Y)-n\sum_{s=1}^3\nabla^2f(e_a,I_se_a)df(I_sY)\Big]
\\
+\frac{n}{n-1}\frac{U(\nabla f,\nabla f)}{|\nabla f|^2}\Big[\nabla^2f(\nabla
f,Y)-\sum_{s=1}^3\nabla^2f(I_s\nabla f,I_sY)\Big] \\
=-\frac1{n-1}Y\Big(\frac{U(\nabla f,\nabla f)}{|\nabla f|^2}\Big)|\nabla
f|^2 +\frac{2n}{n-1}\Big[\frac{n-1}{n+2}fU(Y,\nabla
f)-\sum_{s=1}^3df(\xi_s)U(I_sY,\nabla f)\Big] \\
+\frac{2n}{n-1}fU(Y,\nabla f)+\frac{2n}{n-1}\sum_{s=1}^3df(\xi_s)
U(I_sY,\nabla f) -\frac{4n^2-2}{n-1}\Big[fU(Y,\nabla
f)+\sum_{s=1}^3df(\xi_s) U(I_sY,\nabla f)\Big].
\end{multline}
Thus, from \eqref{divu0} we obtain
\begin{multline}  \label{divu00}
|\nabla f|^2(\nabla_{e_a}U)(e_a,Y)=-\frac1{n-1}Y\Big(\frac{U(\nabla f,\nabla
f)}{|\nabla f|^2}\Big)|\nabla f|^2 \\
-\frac{4n^2+6n}{n+2}fU(Y,\nabla f)-\frac{4n^2-2n}{n-1}\sum_{s=1}^3df(%
\xi_s)U(I_sY,\nabla f).
\end{multline}
A substitution of \eqref{divt00}, \eqref{divu00} and
\begin{equation*}
|\nabla f|^2Y(S)=\frac{4n+4}{n-1}Y\Big(\frac{U(\nabla f,\nabla f)}{|\nabla
f|^2}\Big)|\nabla f|^2
\end{equation*}
in \eqref{conbi2} implies
\begin{equation}  \label{yu}
Y\Big(\frac{U(\nabla f,\nabla f)}{|\nabla f|^2}\Big)|\nabla f|^2=\frac{2n-2}{%
n+2}fU(Y,\nabla f)-2\sum_{s=1}^3df(\xi_s)U(I_sY,\nabla f).
\end{equation}
The equalities \eqref{divu00} and \eqref{yu} yield
\begin{equation}  \label{yu1}
|\nabla f|^2(\nabla_{e_a}U)(e_a,Y)= -\frac{2(n+1)(2n+1)}{n+2}fU(Y,\nabla f)-{2(2n+1)}\sum_{s=1}^3df(\xi_s)U(I_sY,\nabla f).
\end{equation}
We calculate from \eqref{t000} using \eqref{eq7}, \eqref{t0u} and \eqref{yu} that
\begin{multline*}
|\nabla
f|^2(\nabla_ZT^0)(X,Y)-2fdf(Z)T^0(X,Y)-2\sum_{s=1}^3df(\xi_s)df(I_sZ)T^0(X,Y)
\\
= \frac{2n-2}{n+2}fdf(Z)T^0(X,Y)-2\sum_{s=1}^3df(\xi_s)df(I_sZ)T^0(X,Y) \\
-\frac{2n}{n-1}\frac{U(\nabla f,\nabla f)}{|\nabla f|^2} f\Big[-3df(Y)g(X,Z)
+\sum_{s=1}^3df(I_sY)\omega_s(X,Z)\Big] \\
-\frac{2n}{n-1}\frac{U(\nabla f,\nabla f)}{|\nabla f|^2} f\Big[%
-3df(X)g(Y,Z)+\sum_{s=1}^3df(I_sX)\omega_s(Y,Z)\Big] \\
-\frac{2n}{n-1}\frac{U(\nabla f,\nabla f)}{|\nabla f|^2} df(\xi_1)\Big[%
3df(Y)\omega_1(X,Z)+df(I_1Y)g(X,Z)-df(I_2Y)\omega_3(X,Z)+df(I_3Y)%
\omega_2(X,Z)\Big] \\
-\frac{2n}{n-1}\frac{U(\nabla f,\nabla f)}{|\nabla f|^2} df(\xi_1)\Big[%
3df(X)\omega_1(Y,Z)+df(I_1X)g(Y,Z)-df(I_2X)\omega_3(Y,Z)+df(I_3X)%
\omega_2(Y,Z)\Big] \\
-\frac{2n}{n-1}\frac{U(\nabla f,\nabla f)}{|\nabla f|^2} df(\xi_2)\Big[%
3df(Y)\omega_2(X,Z)+df(I_2Y)g(X,Z)+df(I_1Y)\omega_3(X,Z)-df(I_3Y)%
\omega_1(X,Z)\Big] \\
-\frac{2n}{n-1}\frac{U(\nabla f,\nabla f)}{|\nabla f|^2} df(\xi_2)\Big[%
3df(X)\omega_2(Y,Z)+df(I_2X)g(Y,Z)+df(I_1X)\omega_3(Y,Z)-df(I_3X)%
\omega_1(Y,Z)\Big] \\
-\frac{2n}{n-1}\frac{U(\nabla f,\nabla f)}{|\nabla f|^2} df(\xi_3)\Big[%
3df(Y)\omega_3(X,Z)+df(I_3Y)g(X,Z)-df(I_1Y)\omega_2(X,Z)+df(I_2Y)%
\omega_1(X,Z)\Big] \\
-\frac{2n}{n-1}\frac{U(\nabla f,\nabla f)}{|\nabla f|^2} df(\xi_3)\Big[%
3df(X)\omega_3(Y,Z)+df(I_3X)g(Y,Z)-df(I_1X)\omega_2(Y,Z)+df(I_2X)%
\omega_1(Y,Z)\Big].
\end{multline*}
The last equality yields \eqref{nablat0a}. Finally, equation \eqref{vajno} follows from  \eqref{u000} using \eqref{eq7}, \eqref{t0u} and \eqref{yu}. \end{proof}

In the next, key step of the proof, where we show that the torsion tensor vanishes, we shall use the following particular cases of Lemma \ref{l:div torsion}. For $Z=\nabla f$,  \eqref{nablat0a} gives
\begin{multline}  \label{nablat00}
(\nabla_{\nabla f}T^0)(X,Y)=\frac{2n-2}{n+2}fT^0(X,Y)+df(\xi_1)\Big[%
T^0(I_1X,Y)+T^0(X,I_1Y)\Big] \\
+df(\xi_2)\Big[T^0(I_2X,Y)+T^0(X,I_2Y)\Big]+df(\xi_3)\Big[%
T^0(I_3X,Y)+T^0(X,I_3Y)\Big].
\end{multline}
 Similarly, letting $Z=I_{i}\nabla f$ in \eqref{nablat0a} we obtain
\begin{multline}  \label{nablat01}
(\nabla _{I_i\nabla f}T^0)(X,Y)=2df(\xi _i)T^0(X,Y)+f\Big[ %
T^{0}(I_{i}X,Y)+T^{0}(X,I_{i}Y)\Big] \\
-df(\xi _{j})\Big[T^{0}(I_{k}X,Y)+T^{0}(X,I_{k}Y)\Big]+df(\xi _{k})\Big[ %
T^{0}(I_{j}X,Y)+T^{0}(X,I_{j}Y)\Big].
\end{multline}
The substitution of  $Y=\gr$ in \eqref{divt00} taking into account\eqref{tu} and Lemma \ref{l:Lichn identities} gives
\begin{equation}\label{e:divt00}
|\nabla f|^{2}(\nabla _{e_{a}}T^{0})(e_{a},\gr)=\frac{12n(n+1)(2n+1)}{(n+2)(n-1)}%
fU(\gr,\nabla f)
\end{equation}%
while the substitution $Z=e_a, X=I_ie_a,Y=I_i\nabla f$ in \eqref{nablat0a} and
\eqref{t0us} gives
\begin{equation}  \label{nooo}
|\nabla f|^{2}(\nabla _{e_{a}}T^{0})(I_{i}e_{a},I_{i}\nabla f)=- \frac{4n\left(
n+1\right)(2n+1)}{\left( n+2\right) \left( n-1\right) }fU(\nabla
f,\nabla f).
\end{equation}
Finally, letting $Z=\nabla f,I_{s}\nabla f$ in \eqref{vajno} shows  the next
equality
\begin{equation}
(\nabla _{\nabla f}U)(X,Y)=\frac{2n-2}{n+2}fU(X,Y),\qquad (\nabla
_{I_{s}\nabla f}U)(X,Y)=2df(\xi _{s})U(X,Y).  \label{nablau0}
\end{equation}

\subsection{Vanishing of the torsion }\label{ss:qc-einstein}

In this section we show the vanishing of the torsion, $T^0=U=0$, by
calculating in two ways the mixed third covariant derivatives of a function satisfying \eqref{eq7}.
\begin{lemma}\label{l:qc-Einstein}
If $M$ satisfies the assumptions of Theorem \ref{main2}, then the torsion tensor vanishes, $T^0=0$, $U=0$, i.e., $M$ is a qc-Einstein manifold.
\end{lemma}
The proof occupies the remaining part of this sub-section.
\subsubsection{The Ricci identities}
We are going to use the sixth line in \eqref{boh2}. A substitution of the contracted
Bianchi identity \eqref{conbi2} in the second formula of \eqref{d3n6}
gives
\begin{multline}
(2n+1)\rho _{i}(\xi _{j},I_{k}X)=-(2n+1)\rho_{i}(\xi_{k},I_{j}X) \\
=-\frac{1}{4}[(\nabla _{e_{a}}T^{0})(e_{a},X)+(\nabla
_{e_{a}}T^{0})(I_{i}e_{a},I_{i}X)]+\frac{n}{n-1}(\nabla _{e_{a}}U)(e_{a},X).
\label{rhoxijk}
\end{multline}
Let $Z=\nabla f$ in \eqref{d3n5}, and then substitute the obtained equality in the
sixth formula of \eqref{boh2}, after which  use \eqref{rhoxijk} in order to see
\begin{multline}  \label{RicciSpecial1}
\nabla^3f(\xi_i,X,Y)-\nabla^3f(X,Y,\xi_i) \\
=-\nabla^2f(T(\xi_i,X),Y)-\nabla^2f(X,T(\xi_i,Y))-df((\nabla_XT)(\xi_i,Y))+(%
\nabla_XU)(I_iY,\nabla f) \\
+\frac{1}{4}\left[(\nabla_YT^0)(I_i\nabla f,X)+(\nabla_YT^0)(\nabla f,I_iX)%
\right]-\frac{1}{4}\left[(\nabla_{\nabla f}T^0)(I_iY,X)+(\nabla_{\nabla
f}T^0)(Y,I_iX)\right] \\
+\frac{1}{2n+1}\left[-\frac{1}{4}(\nabla_{e_a}T^0)\left[(e_a,I_k\nabla
f)-(I_ke_a,\nabla f)\right]+\frac{n}{n-1}(\nabla_{e_a}U)(e_a,I_k\nabla f)%
\right]\omega_j(X,Y) \\
-\frac{1}{2n+1}\left[-\frac{1}{4}(\nabla_{e_a}T^0)\left[(e_a,I_j\nabla
f)-(I_je_a,\nabla f)\right]+\frac{n}{n-1}(\nabla_{e_a}U)(e_a,I_j\nabla f)%
\right]\omega_k(X,Y) \\
+\frac{1}{2n+1}\left[\frac{1}{4}(\nabla_{e_a}T^0)\left((e_a,I_kY)-(I_ke_a,Y)%
\right)-\frac{n}{n-1}(\nabla_{e_a}U)(e_a,I_kY)\right]df(I_jX) \\
-\frac{1}{2n+1}\left[\frac{1}{4}(\nabla_{e_a}T^0)\left((e_a,I_jY)-(I_je_a,Y)%
\right)-\frac{n}{n-1}(\nabla_{e_a}U)(e_a,I_jY)\right]df(I_kX) \\
+\frac{1}{2n+1}\left[\frac{1}{4}(\nabla_{e_a}T^0)\left((e_a,I_kX)-(I_ke_a,X)%
\right)-\frac{n}{n-1}(\nabla_{e_a}U)(e_a,I_kX)\right]df(I_jY) \\
-\frac{1}{2n+1}\left[\frac{1}{4}(\nabla_{e_a}T^0)\left((e_a,I_jX)-(I_je_a,X)%
\right)-\frac{n}{n-1}(\nabla_{e_a}U)(e_a,I_jX)\right]df(I_kY). \\
\end{multline}
Note that from \eqref{need1} we have
\begin{equation}  \label{need2}
T(\xi_i,X,Y)=-\frac{1}{4}\left[T^0(I_iX,Y)+T^0(X,I_iY)\right]-U(X,I_iY).
\end{equation}
Differentiating the above formula we find, applying \eqref{xider},
\begin{equation}  \label{dfnabla}
df\left( \left( \nabla _{X}T\right) \left( \xi _{i},Y\right) \right) =-\frac{%
1}{4}\left[ (\nabla _{X}T^{0})(I_{i}Y,\nabla f)+(\nabla
_{X}T^{0})(Y,I_{i}\nabla f)\right] +(\nabla _{X}U)(I_{i}Y,\nabla f).
\end{equation}
Using \eqref{eq7}, \eqref{need2}, \eqref{dfnabla} and the properties of
torsion tensor listed in \eqref{propt}, we obtain from \eqref{RicciSpecial1}
\begin{multline}
\nabla ^{3}f(\xi _{i},X,Y)-\nabla ^{3}f(X,Y,\xi _{i})  \label{e:D3Ricci} \\
=\frac{1}{4}\left[ (\nabla _{X}T^{0})(I_{i}Y,\nabla f)+(\nabla
_{X}T^{0})(Y,I_{i}\nabla f)\right] +\frac{1}{4}\left[ (\nabla
_{Y}T^{0})(I_{i}X,\nabla f)+(\nabla _{Y}T^{0})(X,I_{i}\nabla f)\right] \\
-\frac{1}{4}\left[ (\nabla _{\nabla f}T^{0})(X,I_{i}Y)+(\nabla _{\nabla
f}T^{0})(I_{i}X,Y)\right] -\frac{1}{2}\left[ T^{0}(I_{i}X,Y)+T^{0}(X,I_{i}Y)%
\right] f \\
+\left[ \frac{1}{2}\left( T^{0}(X,I_{k}Y)-T^{0}(I_{k}X,Y)\right)
+2U(X,I_{k}Y)\right] df(\xi _{j}) \\
+\left[ \frac{1}{2}\left( T^{0}(I_{j}X,Y)-T^{0}(X,I_{j}Y)\right)
+2U(I_{j}X,Y)\right] df(\xi _{k}) \\
+\frac{1}{2n+1}\left[ -\frac{1}{4}(\nabla _{e_{a}}T^{0})\left[
(e_{a},I_{k}\nabla f)-(I_{k}e_{a},\nabla f)\right] +\frac{n}{n-1}(\nabla
_{e_{a}}U)(e_{a},I_{k}\nabla f)\right] \omega _{j}(X,Y) \\
-\frac{1}{2n+1}\left[ -\frac{1}{4}(\nabla _{e_{a}}T^{0})\left[
(e_{a},I_{j}\nabla f)-(I_{j}e_{a},\nabla f)\right] +\frac{n}{n-1}(\nabla
_{e_{a}}U)(e_{a},I_{j}\nabla f)\right] \omega _{k}(X,Y) \\
+\frac{1}{2n+1}\left[ \frac{1}{4}(\nabla _{e_{a}}T^{0})\left(
(e_{a},I_{k}Y)-(I_{k}e_{a},Y)\right) -\frac{n}{n-1}(\nabla
_{e_{a}}U)(e_{a},I_{k}Y)\right] df(I_{j}X) \\
-\frac{1}{2n+1}\left[ \frac{1}{4}(\nabla _{e_{a}}T^{0})\left(
(e_{a},I_{j}Y)-(I_{j}e_{a},Y)\right) -\frac{n}{n-1}(\nabla
_{e_{a}}U)(e_{a},I_{j}Y)\right] df(I_{k}X) \\
+\frac{1}{2n+1}\left[ \frac{1}{4}(\nabla _{e_{a}}T^{0})\left(
(e_{a},I_{k}X)-(I_{k}e_{a},X)\right) -\frac{n}{n-1}(\nabla
_{e_{a}}U)(e_{a},I_{k}X)\right] df(I_{j}Y) \\
-\frac{1}{2n+1}\left[ \frac{1}{4}(\nabla _{e_{a}}T^{0})\left(
(e_{a},I_{j}X)-(I_{j}e_{a},X)\right) -\frac{n}{n-1}(\nabla
_{e_{a}}U)(e_{a},I_{j}X)\right] df(I_{k}Y). \\
\end{multline}
For $X=I_{i}\nabla f$, $Y=\nabla f$ equation \eqref{e:D3Ricci} together with %
\eqref{nablat00}, \eqref{nablat01}, \eqref{vvv1} and \eqref{t02} imply
\begin{equation}
\nabla ^{3}f(\xi_i,I_i\nabla f,\nabla f)-\nabla ^{3}f(I_{i}\nabla f,\nabla f,\xi _{i})=0.  \label{zero0}
\end{equation}

On the other hand,  a subtraction of \eqref{bixi1} from \eqref{nab3xi} { with  $A=\xi_i$} gives
\begin{multline}  \label{dixiii}
\nabla ^{3}f(\xi _{i},X,Y)-\nabla ^{3}f(X,Y,\xi _{i}) =\frac{n+1}{2n+1}\Big[(\nabla_XT^0)(I_iY,\nabla f)+(\nabla_XT^0)(Y,I_i\nabla f)\Big]\\
+\frac{4n}{(2n+1)(n-1)}(\nabla_XU)(I_iY,\nabla f) -f\frac{n+1}{2n+1}\Big[T^0(X,I_iY)+T^0(I_iX,Y)\Big]+\frac{4n}{%
(2n+1)(n-1)}fU(X,I_iY) \\
-df(\xi_i)\frac{n+1}{2n+1}\Big[T^0(I_iX,I_iY)-T^0(X,Y)\Big]+\frac{4n}{%
(2n+1)(n-1)}df(\xi_i)U(X,Y)\\
-df(\xi_j)\frac{n+1}{2n+1}\Big[T^0(I_jX,I_iY)+T^0(I_kX,Y)\Big]
+\frac{4n%
}{(2n+1)(n-1)}df(\xi_j)U(X,I_kY) \\
-df(\xi_k)\frac{n+1}{2n+1}\Big[T^0(I_kX,I_iY)-T^0(I_jX,Y)\Big]-\frac{4n%
}{(2n+1)(n-1)}df(\xi_k)U(X,I_jY)\\
-\sum_{s=1}^3\left[ \nabla ^{2}f(\xi _{i},\xi _{s})+f\right] \omega _{s}(X,Y).
\end{multline}
Letting $X=I_i\nabla f, Y=\nabla f$ in \eqref{dixiii} and then applying %
\eqref{vvv1} and \eqref{t02} we obtain
\begin{multline}  \label{zero1}
\nabla ^{3}f(\xi _{i},I_i\gr ,\gr)-\nabla ^{3}f(I_i\gr,\gr,\xi _{i}) =\frac{2(n+1)}{2n+1}(\nabla_{I_i\nabla f}T^0)(I_i\nabla f,\nabla f)\\
+ \frac{4n}{(2n+1)(n-1)}(\nabla_{I_i\nabla f}U)(I_i\nabla f,\nabla f) -\frac{n+1}{2n+1}f\Big[T^0(I_i\nabla f,I_i\nabla f)-T^0(\nabla f,\nabla f)\Big]\\
+\frac{4n}{(2n+1)(n-1)}fU(\nabla f,\nabla f)+\left[ \nabla^{2}f(\xi _{i},\xi _{i})+f\right] |\nabla f|^2.
\end{multline}
Using \eqref{nablat01}, \eqref{nablau0} as well as \eqref{tu} in %
\eqref{zero1} we conclude
\begin{equation}  \label{zero2}
\nabla ^{3}f(\xi _{i},I_i\gr ,\gr)-\nabla ^{3}f(I_i\gr,\gr,\xi _{i}) = \frac{4n}{n-1}fU(\nabla f,\nabla f)+\left[
\nabla ^{2}f(\xi _{i},\xi _{i})+f\right] |\nabla f|^2.
\end{equation}
The formula for the last term is given in \eqref{xixi} to whose right-hand side  we apply \eqref{e:divt00},
\eqref{nooo} and  \eqref{yu1} in order to obtain
\begin{equation}  \label{zero3}
\nabla ^{2}f(\xi _{i},\xi _{i})+f=-\frac{2\left( n+1\right) \left(
2n+1\right) }{\left( n+2\right) \left( n-1\right) }f\frac{U(\nabla f,\nabla
f)}{|\nabla f|^{2}}
\end{equation}
Now \eqref{zero3} applied to \eqref{zero2} allows us to conclude
\begin{equation}  \label{zero4}
\nabla ^{3}f(\xi _{i},I_i\gr ,\gr)-\nabla ^{3}f(I_i\gr,\gr,\xi _{i})= \frac{2}{n+2}fU(\nabla f,\nabla f).
\end{equation}
Comparing \eqref{zero0} and \eqref{zero4} we obtain $fU(\nabla f,\nabla f)=0$,
which implies $U(\nabla f,\nabla f)=0$ taking into account  Remark~\ref{eell}. Hence, $%
T^0=U=0 $ due to \eqref{t000} and \eqref{u000}. This completes the proof of Lemma \ref{l:qc-Einstein}.

\subsection{The Riemannian Hessian.}

Here we show that if $T^0=U=0$ the equality \eqref{eq7} implies that the
Riemannian Hessian satisfies \eqref{clasob} and therefore the
manifold is the standard sphere due to the Obata's theorem.

\begin{lemma}\label{l:Riem hessian}
Let $(M,\eta,g,\mathbb{Q})$ be a qc-Einstein manifold, $T^0=U=0$, of dimension $4n+3>7$. Let $h$ be the associated Riemannian metric \eqref{riem1}.
If  $f$  is   smooth function whose horizontal Hessian satisfies \eqref{eq7}, then the Riemannian Hessian of $f$ with respect to the metric $h$ satisfies \eqref{clasob}.
\end{lemma}

\begin{proof}
Taking into account \eqref{lcbi} we have the following formula relating the Hessian with respect to the Levi-Civita and the  Biquard connections
\begin{equation}  \label{hes11}
(\nabla^h)^2f(A,B)=\nabla^2f(A,B)+\frac12\Big[%
h(T(A,B),df)-h(T(B,df),A)+h(T(df,A),B)\Big], A,B\in\Gamma(TM).
\end{equation}
From \eqref{hes11}, \eqref{torha} and \eqref{eq7} it follows that
\begin{equation}  \label{hes12}
(\nabla^h)^2f(X,Y)=-fh(X,Y).
\end{equation}
Let us recall that  a qc-Einstein manifold, $T^0=U=0$,  has integrable vertical space \cite{IMV} thus the fourth line in \eqref{sixtyfour} shows
\begin{equation}  \label{ijk}
h(T(\xi_s,\xi_t),X)=0,
\end{equation}
Now, using  \eqref{vvvv52} with $T^0=U=0$, we calculate from %
\eqref{hes11}
\begin{multline}  \label{hes13}
(\nabla^h)^2f(X,\xi_i)=df(I_iX)+\frac12h(T(X,\xi_i),\nabla
f)-\frac12h(T(\xi_i,\nabla f),X) -\frac12h(T(\xi_i,\sum_{s=1}^3df(\xi_s)\xi_s,X)\\
+\frac12h(T(\nabla
f,X),\xi_i)+\frac12h(T(\sum_{s=1}^3df(\xi_s)\xi_s,X),\xi_i)
=df(I_iX)+\omega_i(\nabla f,X)=0,
\end{multline}
{taking into account \eqref{ijk} and the properties of the torsion \eqref{need1} and \eqref{torha}.} A similar argument shows the identity
\begin{equation}  \label{hes14}
(\nabla^h)^2f(\xi_i,\xi_i)=\nabla^2f(\xi_i,\xi_i)=-f,
\end{equation}
where we have used \eqref{xixi} taken with $T^0=U=0$.

Finally, we have to show $(\nabla^h)^2f(\xi_i,\xi_j)=0$. The trace with
respect to $X=e_a,Y=I_je_a$ in \eqref{ricxi} together with the second
equality in \eqref{sixtyfour} and the condition $T^0=U=0$ yields
\begin{equation}  \label{hij}
\nabla^2f(\xi_j,\xi_i)=(1-S)df(\xi_k).
\end{equation}
Now, the equality \eqref{hes11} together with the fourth equality in %
\eqref{sixtyfour}, \eqref{ijk} and \eqref{hij} imply
\begin{equation}  \label{hes15}
(\nabla^h)^2f(\xi_i,\xi_j)=(1-S)df(\xi_k)+\frac12Sdf(\xi_k)=(1-\frac12S)df(%
\xi_k)=0,
\end{equation}
since \eqref{vvvv4} shows $S=2$ in the case $T^0=U=0$.
\end{proof}

{At this point, applying the Obata theorem we conclude that our manifold is isometric to the unit sphere. In order to show that it is qc-equivalent to the sphere we shall use a Liouville-type result in the quaternionic contact case which we present next.}

\subsection{Quaternionic contact conformally flat manifolds and the Liouville theorem}\label{ss:qc conf flat}
We start by recalling some basic definitions and facts.

\begin{dfn}[{\cite[Section 7.2]{IMV}}]\label{d:3-ctct auto}
A diffeomorphism $F$ of a qc manifold $(M,\eta)$ is
called a conformal quaternionic contact automorphism
(abbr. conformal qc-automorphism or qc-conformal map) if $F$ preserves the qc structure,
i.e. $F^*\eta=\mu\Psi\cdot\eta,
$ for some positive smooth function $\mu$ and matrix $\Psi\in
SO(3)$ with smooth functions as entries, where
$\eta=(\eta_1,\eta_2,\eta_3)^t$ is a local 1-form considered as an
element of $\mathbb R^3$. If $\mu=const$ we call $F$ a quaternionic contact homothety (abbr.  qc-homothety). Finally, if $\mu=1$ then $F$ is called a quaternionic contact automorphism (abbr.  qc-automorphism).
\end{dfn}
The above definition extends in the obvious manner to a map between two qc manifolds  giving the notion of a qc-conformal map.
The qc-conformal curvature tensor $W^{qc}$, introduced in \cite{IV}, is the obstruction for a qc structure to be locally qc-conformal to the flat structure on the quaternionic Heisenberg group $\boldsymbol{G\,(\mathbb{H})}$. A qc-conformally flat structure is also locally qc-conformal to the
standard 3-Sasaki sphere due to the local qc-conformal equivalence of the
standard 3-Sasakian structure on the $(4n+3)$-dimensional sphere and the
quaternionic Heisenberg group \cite{IMV,IV} via the Cayley transform whose definition we recall below, see also \cite[Section 2.3.1]{IV2} for the history and references of the Cayley transform on groups of Heisenberg type.
In view of the uniqueness of the possible associated almost complex structures \cite{Biq1}, (see also  \cite[Lemma 2.2]{IMV}), a conformal quaternionic contact automorphism will also preserve, unlike the situation in the CR case,  the associated  almost complex structures and  the conformal class $[g]$ of the horizontal metric $g$  on $H$.

We shall use the following model of the  quaternionic Heisenberg group $\QH$, see \cite[Section 5.2]{IMV} or \cite[Section 4.3.4]{IV2}. Define
$\QH\ =\ \Hn\times\text {Im}\, \mathbb{H}$ with the group law given by
\begin{equation}\label{e:Heise multipl}
 (q_o, \omega_o)\circ(q, \omega)\ =\ (q_o\ +\ q, \omega\ +\ \omega_o\ +
\ 2\ \text {Im}\  q_o\, \bar q),
\end{equation}
where $q,\ q_o\in\Hn$, $\omega, \omega_o\in \text {Im}\, \mathbb{H}$ and  $\mathbb{H}$ is  the space of quaternions.  We shall denote with $(q, \omega)$ the elements of $\QH$. The standard 3-contact form
$\tilde\Theta=(\tilde\Theta_1,\ \tilde\Theta_2, \ \tilde\Theta_3)$ written as a purely imaginary quaternion valued 1-form is defined by
\begin{equation}\label{e:Heisenberg ctct forms}
\tilde\Theta\  =\ \frac 12\left (
d\omega \ - \ q \cdot d\bar q \ + \ dq\, \cdot\bar q\right ).
\end{equation}

Let ${S^{4n+3}}\ =\ \{\abs{q'}^2+\abs{p'}^2=1 \}\subset \Hn\times\mathbb{H}$ be the unit sphere in the $(n+1)$-dimensional quaternion space  equipped with the qc structure defined by   the standard contact form $\tilde\eta$ on the sphere,
\begin{equation}\label{e:stand cont form on S}
\tilde\eta\ =\ dq'\cdot \bar q'\ +\ dp'\cdot \bar p'\ -\ q'\cdot d\bar q' -\ p'\cdot d\bar p'.
\end{equation}
  Note that $\tilde\eta$ is twice the standard 3-Sasakain form on the sphere.  Following  \cite{IMV}, see also \cite{IV2},  we identify the quaternionic Heiseneberg group $\QH$  with the boundary  of a Siegel domain in $\Hn\times\mathbb{H}$,
\[
\QH\ =\ \{ (q,p)\in \Hn\times\mathbb{H}\ :\ \Re {\ p}\ =\ \abs{q}^2 \},
\]
 using the map $(q, \omega)\mapsto (q,\abs{q}^2 - \omega)$.  Under this identification   we have
\begin{equation}\label{e:Siegel ctct forms}
\tilde{\Theta}\ =\  \frac 12\big [- dp\ +\ 2dq\cdot\bar {q}\Big ]\ =\ \frac 14\Big [ (d\bar p\ -\ d p)\ +\ 2dq\cdot\bar {q}\ -\ 2q\cdot d\bar q \Big] . \end{equation}
The Cayley transform, see \cite{Ko1} and \cite{CDKR}, is the map   $\mathcal{C}:S^{4n+3}\setminus \{(-1,0) \} \rightarrow \QH$, $(-1,0)\in \Hn\times\mathbb{H}$, from the
sphere ${S^{4n+3}}\ =\ \{\abs{q'}^2+\abs{p'}^2=1 \}\subset \Hn\times\mathbb{H}$ minus a point to the Heisenberg group $
\QH$, with $\mathcal{C}$ defined by
\begin{equation}\label{d:Cayley}
 (q, p)\ =\ \mathcal{C}\ \Big ((q', p')\Big), \qquad
q\ =\ (1+p')^{-1} \ q', \qquad p\ =\ (1+p')^{-1} \ (1-p')
\end{equation}
and with an inverse map $(q', p')\ =\ \mathcal{C}^{-1}\Big ((q, p)\Big)$ given by
$
q' =  2(1+p)^{-1} \, q$, $  p' =(1+p)^{-1} \, (1-p)$.
The Cayley transform is a conformal  quaternionic contact diffeomorphism between the quaternionic
Heisenberg group with its standard quaternionic contact structure $\tilde\Theta$ and the sphere minus a point
with its standard structure $\tilde\eta$, \cite{IMV}.  In fact, by \cite[Section 8.3]{IMV} we have
\begin{equation*}
 \lambda\ \cdot (\mathcal{C}^{-1})^*\, \tilde\eta\ \cdot \bar\lambda\ =\ \frac
{8}{\abs{1+p\, }^2}\, \tilde\Theta.
\end{equation*}
\noindent where $\lambda\ = {\abs {1+p\,}}\, {(1+p)^{-1}}$ is a unit quaternion and $\tilde\eta$ is defined in \eqref{e:stand cont form on S}.

Besides the Cayley transform, we shall need the generalization of the Euclidean inversion transformation to the qc setting.  We recall that in \cite{Ko1} Kor\'anyi introduced such an inversion and an analogue of the Kelvin transform on the Heisenberg group, which were later generalized in \cite{CK} and \cite{CDKR} to all groups of Heisenberg type. The inversion and Kelvin transform enjoy useful properties in the case of the four groups of Iwasawa type of which $\QH$ is a particular case.   For our goals it is necessary to show that the inversion on $\QH$ is a qc-conformal map. In order to prove this fact we shall represent the inversion as the composition of two Cayley transforms, see \cite{IMV1,IV2} where  the seven dimensional case was used.   Let $P_1=(-1,0)$ and $P_2=(1,0)$ be correspondingly the 'south' and 'north'
poles of the unit sphere  ${S^{4n+3}}\ =\ \{\abs{q}^2+\abs{p}^2=1 \}\subset \Hn\times\mathbb{H}$.  Let $\mathcal{C}_1$ and $\mathcal{C}_2$ be the corresponding Cayley transforms defined, respectively, on $S^{4n+3}\setminus\{P_1\}$ and $S^{4n+3}\setminus\{P_2\}$. Note that $\mathcal{C}_1$ was defined in \eqref{d:Cayley}%
, while $\mathcal{C}_2$ is given by $(q, p)\ =\ \mathcal{C}_2\
\Big ((q', p')\Big)$,
\begin{equation}  \label{d:2nd Cayley}
q\ =\ -(1-p')^{-1} \ q',\quad p\ =\ (1-p')^{-1} \ (1+p'), \qquad  (q',p')\in S^{4n+3}\setminus\{P_2\}
\end{equation}
The inversion on the quaternionic Heisenberg group (with center the origin) with respect to the unit gauge sphere   is the map
\begin{equation}\label{d:inversion}
\sigma=\mathcal{C}_2\circ\mathcal{C}_1^{-1}:\QH\setminus \{ (0,0)\}
\rightarrow\QH\setminus \{ (0,0)\}.
\end{equation}
  In particular, $\sigma\ =\ \mathcal{C}_2\circ\mathcal{C}_1^{-1} $ is {an involution}
on the group.   A small calculation shows that $\sigma$ is given by the formula (in the Siegel model)
\begin{equation*}
q^*\ =\ -p^{-1}\, q,\qquad p^*\ =\ p^{-1},
\end{equation*}
or, equivalently, in the direct product model $\boldsymbol{G\,(\mathbb{H})}$
\begin{equation*}
q^*\ =\ -(|q|^2-\omega)^{-1}\, q, \qquad \omega^*\ =\ -\frac {\omega%
}{|q|^4+|\omega|^2}.
\end{equation*}
Recalling  \eqref{e:Siegel ctct forms} and  \eqref{e:Heisenberg ctct forms} it follows
\begin{equation}
\begin{aligned}
\sigma ^*\ \Theta\ =\ \frac {1}{|p|^2}
\,\bar\mu\, \Theta\, \mu, \qquad \mu\ =\ \frac {p}{|p|},\qquad \text{ (in the Siegel model)}\\
\sigma ^*\ \Theta\ =\ \frac {1}{|q|^4+|\omega|^2}
\,\bar\mu\, \Theta\, \mu, \qquad \mu\ =\ \frac {|q|^2+\omega}{\left ( |q|^4+|\omega|^2\right)^{1/2}},\qquad \text{ (in the product model)}.
\end{aligned}
\end{equation}
Thus, we have the following fundamental fact.
\begin{lemma}\label{l:inversion qc}
The inversion transformation \eqref{d:inversion} is a qc-conformal transformation on the quaternionic Heisenberg group.
\end{lemma}
As usual, using the dilations and translations on the group, it is a simple matter to define an inversion   with respect to any gauge ball.  We recall that the gauge norm of a point $(q,\omega)\in \QH$ is $N(q,\omega)=\left ( |q|^4+|\omega|^2 \right)^{1/4}$ which allows to define a distance on the group using the translation structure.

After these preliminary facts, our next goal is to prove a version of Liouville's theorem in the case of the quaternionic Heisenberg group and the 3-Sasakian sphere equipped with their standard qc structures. In the Euclidean  case Liouville  \cite{Liu1850}, \cite{Liu1850b} showed that every sufficiently smooth conformal map ($\emph{C}^4$ in fact) between two connected open sets of the Euclidean space $\mathbb{R}^3$  is necessarily a restriction of a M\"obius transformation. The latter is the group generated by translations, dilations  and inversions of the extended Euclidean space obtained by adding an ideal point at infinity.  Liouville's result generalizes easily to any dimension $n>3$. Subsequently, Hartman \cite{Har58} gave a proof requiring only $\mathcal{C}^1$ smoothness of the conformal map, see also  \cite{Ne60},  \cite{BoIw82}, \cite{Ja91},    \cite{IwMa98} and  \cite{Fr03} for other proofs.   A CR version of Liouville's result can be found in \cite{Ta62} and \cite{Al74}. Thus, a smooth CR diffeomorphism between two connected open subsets of the $2n+1$ dimensional  sphere is  the restriction of an element from the isometry group $SU(n+1,1)$ of the unit ball equipped with the complex hyperbolic metric. The proof of Alexander \cite{Al74} relies on the extension property of a smooth CR map to a biholomorphism. Tanaka, see also \cite{Poi07}, \cite{Car32,Car32a} and \cite{ChM}, in his study of pseudo-conformal equivalence between analytic real hypersurfaces of complex space showed a more general result
 \cite[Theorem 6]{Ta62} showing that any pseudo-conformal  homeomorphism between connected open sets of the quadric
 \[
 -\sum_{i=1}^r |z_i|^2+\sum_{i=r+1}^n |z_i|^2 =1, \quad (z_1,\dots,z_n)\in \mathbb{C}^n,
 \]
is the restriction of a projective transformation of $P^{n}(\mathbb{C})$.

A new stage began with the introduction of quasiconformal maps, which imposed metric conditions on the maps,   \cite{Ge62} and\cite{Re67} and with the works of Mostow \cite{M} and Pansu \cite{P} .  In particular, in \cite{P} it was shown  that every global 1-quasiconformal map on the sphere at infinity  of each of the hyperbolic metrics is an isometry of the corresponding hyperbolic space.  The local version of the Liouville's property for  1-quasiconformal map of class $\emph{C}^4$ on the Heisenberg group  was settled in \cite{KoRe85} by a  reduction to the CR  result. The optimal regularity question for quasiconformal maps  in much greater generality, including the cases of all Iwasawa type groups, was proven later by Capogna \cite{Cap97}, see also \cite{Ta96} and \cite{CapCow06}.

 A closely related  property is the so called rigidity property of quasiconformal or {multicontact} maps, also referred to as Liouville's property, but where the question is the finite dimensionality of the group of (locally defined) quasiconformal or {multicontact} maps, see \cite{Ya93},   \cite{Rei01},  \cite{CMKR05}, \cite{Ot05}, \cite{Ot08},\cite{Mor09}, \cite{dMOt10}, \cite{lDOt11}.

  In Theorem \ref {l:Liouville} we give  a local  Liouville type property in the setting of a sufficiently smooth qc-conformal maps relying only on the qc geometry.
\subsection{Proof of Theorem \ref{l:Liouville}.}\label{ss:Liouville}
Since in the case $\Sigma={S^{4n+3}}$ there is nothing to prove we shall assume that $\Sigma\not={S^{4n+3}}$.
We shall transfer the analysis to the quaternionic Heisenberg group using the Cayley transform, thereby reducing  to the case of a qc-conformal transformation $\tilde F:\tilde \Sigma \rightarrow \QH$ between two domains of the quaternionic Heisenberg group such that $\Theta={\tilde F}^*\tilde\Theta= \frac{1}{2\mu} \tilde\Theta$ for some positive smooth function $\mu$ defined on the open set $\tilde\Sigma$.  By its definition $\Theta$ is a qc-Einstein structure,  hence the proof of \cite[Theorem 1.1]{IMV} shows that $\mu$ satisfies a system of partial differential equations whose solution is a family of polynomial of degree four. In fact,  with a small change of the parameters in \cite[Theorem 1.1]{IMV}, $\mu$ has the form
\begin{equation}\label{e:Liouville conf factor}
\mu (q,\omega) \ =\ c_0\ \Big [  \big ( \sigma\ +\
 |q+q_0|^2 \big )^2\  +\  |\omega\ +\ \omega_o\ +
\ 2\ \text {Im}\  q_o\, \bar q|^2 \Big ],
\end{equation}
for some fixed $(q_o,\omega_o)\in \QH$ and constants $c_0>0$ and $\sigma\in \mathbb{R}$. A small calculation using \eqref{e:Liouville conf factor} and the Yamabe equation  \cite[(5.8)]{IMV} shows $Scal_{\Theta}=128n(n+2)c_0\sigma $. Since $\Theta$ is qc-conformal  to $\tilde\Theta$ via the map $F$,  it follows that  $\sigma=0$. Hence, $F$ is a composition of a translation, cf. \eqref{e:Heise multipl}, followed by an inversion and a homothety, cf. Lemma \ref{l:inversion qc}.

The above analysis implies that $F$ is the restriction of an element of $PSp(n+1,1)$. This completes the proof of Theorem \ref{l:Liouville}.

Similar to the Riemannian and CR cases Theorem \ref{l:Liouville} and a standard monodromy type argument, see \cite{Ku49} and \cite[Theorem  VI.1.6]{SchYa}, show the validity of the next
\begin{thrm}\label{t:conf sphere}
If $(M,\eta)$ is a simply connected qc-conformally flat manifold of dimension $4n+3$, then there is a qc-conformal immersion $\Phi:M\rightarrow {S^{4n+3}}$, where ${S^{4n+3}}$ is the 3-Sasakian unit sphere in the $(n+1)$-dimensional quaternion space.
\end{thrm}

\subsection{Proof of Theorem \ref{main2}.}
From Lemma \ref{l:qc-Einstein}, Lemma \ref{l:Riem hessian} and the classical Obata theorem it follows  that $(M,h)$ is isometric to the unit sphere in Euclidean space, i.e., there is a diffeomorphism $i:M\rightarrow S^{4n+3}$ such that $h=i^* dx^2$, where $dx^2$ denotes the round metric on $S^{4n+3}$  {which we take to be of constant Riemannian scalar curvature $Scal^h=(4n+3)(4n+2)$. Thus, the curvature tensor $R^h$ of the Levi-Civita connection $\bi^h$ of $h$ is given by
\begin{equation}\label{h}
R^h(A,B,C,D)=h(B,C)h(A,D)-h(B,D)h(A,C).
\end{equation}
The relation between the curvature tensors of the Levi-Civita and the Biquard connection \cite[Corollary~4.13]{IMV} or  \cite[Theorem~4.4.3]{IV2} together with \eqref{h} yield
\begin{multline}\label{hb}
R(X,Y,Z,V)=g(Y,Z)g(X,V)-g(Y,V)g(X,Z)\\
+\sum_{s=1}^3\Big[\omega_s(Y,Z)\omega_s(X,V)-\omega_s(X,Z)\omega_s(Y,V)-2\omega_s(X,Y)\omega_s(Z,V)\Big].
\end{multline}
In the case $T^0=U=0, S=2$, the formula for the qc-conformal curvature tensor given in \cite[Proposition~4.2]{IV}  reads
\begin{multline}\label{qccurv}
W^{qc}(X,Y,Z,V)=\frac14\Big[R(X,Y,Z,V)+\sum_{s=1}^3R(I_sX,I_sY,Z,V)\Big]+g(X,Z)g(Y,V)-g(Y,Z)g(X,V)\\+\sum_{s=1}^3\Big[\omega_s(X,Z)\omega_s(Y,V)-\omega_s(Y,Z)\omega_s(X,V)\Big].
\end{multline}
With  a small calculation we see from \eqref{qccurv}, taking into account  \eqref{hb},  that the qc-conformal curvature tensor vanishes, $W^{qc}=0$  and $(M,g,\eta,\mathbb Q)$  is locally qc-conformal to the sphere due to \cite[Theorem~1.3]{IV}. Hence, taking into account Theorem \ref{t:conf sphere}, $(M,g,\eta,\mathbb Q)$  is qc-conformal to ${S^{4n+3}}$, i.e.,  $\eta=\kappa \Psi  F^*\tilde\eta$ for some diffeomorphsm $F:M\rightarrow S^{4n+3}$. Comparing the metrics we obtain $\kappa=1$ which shows that $M$ is qc-homothetic to  the 3-Sasakian unit sphere in the $(n+1)$-dimensional quaternion space. } This completes the proof of Theorem~\ref{main2}. The proof  of Theorem~\ref{main1} follows as already noted after the statement of Theorem~\ref{main2}.

\section{Appendix}

\subsection{The $P-$form}

Let $(M,g,\mathbb{Q})$ be a compact quaternionic contact manifold of
dimension $4n+3$ and $f$ a smooth function on $M$. We recall the notion of a
$P$-function introduced in \cite{IPV2}

\begin{dfn}\label{d:def P}[\cite{IPV2}]
\begin{enumerate}[a)]
\item For a fixed $f$ we define a one form $P\equiv P_f \equiv P[f]$ on $M$,
which we call the $P-$form of $f$, by the following equation
\begin{multline*}
P_f(X) =\nabla ^{3}f(X,e_{b},e_{b})+\sum_{t=1}^{3}\nabla
^{3}f(I_{t}X,e_{b},I_{t}e_{b})-4nSdf(X)+4nT^{0}(X,\nabla f)-\frac{8n(n-2)}{n-1%
}U(X,\nabla f).
\end{multline*}

\item The $P-$function of $f$ is the function $P_f(\nabla f)$.

\item The $C-$operator is the fourth-order differential operator on $M$
(independent of $f$!) defined by
\begin{equation*}
Cf =-\nabla^* P_f=(\nabla_{e_a} P_f)\,(e_a).
\end{equation*}

\item We say that the $P-$function of $f$ is non-negative if its integral
exists and is non-positive
\begin{equation}  \label{e:non-negative Paneitz}
\int_M f\cdot Cf \, Vol_{\eta}= -\int_M P_f(\nabla f)\, Vol_{\eta}\geq 0.
\end{equation}
If \eqref{e:non-negative Paneitz} holds for any smooth function of compact
support we say that the $C-$operator is non-negative.
\end{enumerate}
\end{dfn}

The $Sp(n)Sp(1)$-invariant decomposition of the horizontal Hessian $%
\nabla^2f $ are given by
\begin{equation}  \label{comp}
\begin{aligned}
(\nabla^2f)_{[3]}(X,Y)=\frac14\Big[\nabla^2f(X,Y)+\sum_{s=1}^3%
\nabla^2f(I_sX,I_sY)\Big]\\
(\nabla^2f)_{[-1]}(X,Y)=\frac14\Big[3\nabla^2f(X,Y)-\sum_{s=1}^3%
\nabla^2f(I_sX,I_sY)\Big]. \end{aligned}
\end{equation}
Let $(\nabla^2f)_{[3][0]}$ be the trace-free part of the 3-component of the
horizontal Hessian,
\begin{equation}  \label{np1}
(\nabla^2f)_{[3][0]}(X,Y)=(\nabla^2f)_{[3]}(X,Y)+\frac1{4n}\triangle fg(X,Y).
\end{equation}
The next local formula, established in \cite{IPV2},
\begin{equation}  \label{panon}
(\nabla_{e_a}(\nabla^2f)_{[3][0]})(e_a,X)=\frac{n-1}{4n}P_f(X).
\end{equation}
implies the non-negativity of the $C-$operator on a compact qc manifold of
dimension at least eleven \cite[Theorem~3.3]{IPV2}. Indeed, using %
\eqref{panon} we have
\begin{equation}  \label{pppp}
\frac{n-1}{4n}\int_Mf.Cf\, Vol_{\eta}=-\frac{n-1}{4n}\int_MP_f(\nabla f)\,
Vol_{\eta}=\int_M|(\nabla^2f)_{[3][0]}|^2\, Vol_{\eta},
\end{equation}
after using an integration by parts and the orthogonality of the components of the
horizontal Hessian.

\subsection{A new proof of Theorem~\protect\ref{mainpan}}

Here we use the non-negativity of the $P$-function $P(\nabla f)$ of a smooth
function $f$ established in \cite[Theorem~3.3]{IPV2} to give a new proof of
Theorem~\ref{mainpan}.

\begin{proof}
Let $f$ be an eigenfunction of the sub-Laplacian with eigenvalue $%
\lambda $, i.e., we assume that \eqref{eig} holds. An integration by parts
gives
\begin{equation}
\int_{M}(\triangle f)^{2}\,\,Vol_{\eta }=\lambda \int_{M}f\triangle
f\,\,Vol_{\eta }=\lambda \int_{M}|\nabla f|^{2}\,\,Vol_{\eta }.
\label{parts}
\end{equation}
We recall the qc-Bochner identity \cite[(4.1)]{IPV1}. Applying the first
equality in \eqref{sixtyfour}, \eqref{need1} and the properties of the
torsion, \eqref{propt}, we can write the qc-Bochner formula \cite[(4.1)]{IPV1}
in the form
\begin{eqnarray}  \label{bohS}
\frac12\triangle |\nabla f|^2=|\nabla^2f|^2-g\left (\nabla (\triangle f),
\nabla f \right )+2(n+2)S|\nabla f|^2+2(n+2)T^0(\nabla f,\nabla f) \\
+2(2n+2)U(\nabla f,\nabla f)+ 4\sum_{s=1}^3\nabla^2f(\xi_s,I_s\nabla f) .
\notag
\end{eqnarray}
One of the key identities which relates the P-function and the qc-Bochner
formula \eqref{bohS} is given by the following equation \cite[Lemma~3.2]{IPV2}
\begin{equation}  \label{e:gr4}
\int_M\sum_{s=1}^3\nabla^2f(\xi_s,I_s\nabla f)\, Vol_{\eta}=\int_M\Big[-%
\frac{1}{4n}P_n(\nabla f)-\frac{1}{4n}(\triangle f)^2-S|\nabla f|^2+\frac{%
(n+1)}{n-1}U(\nabla f,\nabla f)\Big]\, Vol_{\eta}.
\end{equation}
An integration of \eqref{bohS} over the compact $M$, followed by a substitution of \eqref{eig}
and \eqref{e:gr4} in the obtained integral equality, and then a use of the
divergence formula \eqref{div}  give
\begin{eqnarray}  \label{inteq}
0=\int_M\Big[|\nabla^2f|^2-\lambda|\nabla f|^2+2nS|\nabla
f|^2+2(n+2)T^0(\nabla f,\nabla f)+\frac{4n(n+1)}{n-1}U(\nabla f,\nabla f) \\
-\frac{1}{n}P_n(\nabla f)-\frac{1}{n}(\triangle f)^2\Big]\,\,Vol_{\eta } .
\notag
\end{eqnarray}
The latter formula can be written in the form
\begin{eqnarray}  \label{inteq1}
0=\int_M\Big\{|\nabla^2f|^2-\lambda|\nabla f|^2-S|\nabla f|^2+T^0(\nabla
f,\nabla f)-\frac{2(n-2)}{n-1}U(\nabla f,\nabla f)  \notag \\
+\frac{2n+1}{2(n+2)}\Big[2(n+2)S|\nabla f|^2 +\frac{4n^2+14n+12}{2n+1}%
T^0(\nabla f,\nabla f)+\frac{4(n+2)^2(2n-1)}{(n-1)(2n+1)}U(\nabla f,\nabla f)%
\Big] \\
-\frac{1}{n}P_n(\nabla f)-\frac{1}{n}(\triangle f)^2\Big\}\,\,Vol_{\eta}.
\notag
\end{eqnarray}
Now we invoke the next integral identity proved in \cite[Lemma~3.4]{IPV1}
\begin{equation}  \label{2}
\int_M\sum_{s=1}^3\nabla^2f(\xi_s,I_s\nabla f)\, Vol_{\eta}=-\int_M\Big[%
4n\sum_{s=1}^3(df(\xi_s))^2 +\sum_{s=1}^3T(\xi_s,I_s\nabla f,\nabla f)\Big] %
\, Vol_{\eta}.
\end{equation}
From \eqref{2} and \eqref{e:gr4} it follows the equality
\begin{eqnarray}  \label{inteq2}
\int_M\Big[-S|\nabla f|^2+T^0(\nabla f,\nabla f)-\frac{2(n-2)}{n-1}U(\nabla
f,\nabla f)\Big]\,\,Vol_{\eta}=\int_M\Big\{\frac{1}{4n}P_n(\nabla f)+\frac{1%
}{4n}(\triangle f)^2 \\
-\frac{1}{4n}\sum_{s=1}^{3}[g(\nabla^2f,\omega_s)]^2\Big\}\,\,Vol_{\eta} .
\notag
\end{eqnarray}
A substitution of \eqref{inteq2} in \eqref{inteq1} yields
\begin{multline}  \label{intin1}
0=\int_M\Big\{|\nabla^2f|^2-\frac{1}{4n}\Big[(\triangle
f)^2+\sum_{s=1}^{3}[g(\nabla^2f,\omega_s)]^2\Big]-\frac{3}{4n}P_n(\nabla f)
\\
+ \frac{2n+1}{2}\Big[2S|\nabla f|^2 +\frac{4n^2+14n+12}{(2n+1)(n+2)}%
T^0(\nabla f,\nabla f)+\frac{4(n+2)(2n-1)}{(n-1)(2n+1)}U(\nabla f,\nabla f)-%
\frac{\lambda}{n}|\nabla f|^2\Big]\Big\}\,\,Vol_{\eta}.
\end{multline}
The equality \eqref{intin1}, the Lichnerowicz type assumption \eqref{condm-app} and
\eqref{parts} imply the inequality
\begin{multline}  \label{intin}
0\geq\int_M\Big\{|\nabla^2f|^2-\frac{1}{4n}\Big[(\triangle
f)^2+\sum_{s=1}^{3}[g(\nabla^2f,\omega_s)]^2\Big]-\frac{3}{4n}P_n(\nabla f)+%
\frac{2n+1}{2}\Big(\frac{k_0}{n+2}-\frac{\lambda}{n}\Big)|\nabla f|^2\Big\}%
\,\,Vol_{\eta}.
\end{multline}
Note that in the proof of  \eqref{intin} we supposed
implicitly that $n>1$. But it works also for $n=1$, when $U=0$ trivially, we have only to remove the
torsion tensor $U$ (cf. \cite{IPV2}).

Using that $\left\{ \frac{1}{2\sqrt{n}}\omega _{s}\right\} $ is an
orthonormal set in $\Psi _{\lbrack -1]}$ we have
\begin{equation}
|(\nabla ^{2}f)_{[-1]}|^{2}\geq \frac{1}{4n}\sum_{s=1}^{3}\left[ g\left(
\nabla ^{2}f,\omega _{s}\right) \right] ^{2}  \label{coshy1}
\end{equation}%
while a projection on $\left\{ \frac{1}{2\sqrt{n}}g\right\} $ gives
\begin{equation}
|(\nabla ^{2}f)_{[3]}|^{2}\geq \frac{1}{4n}(\triangle f)^{2}.  \label{coshy3}
\end{equation}
Next, using the $Sp(n)Sp(1)$-invariant orthogonal decomposition \eqref{comp}
of horizontal Hessian and the estimates \eqref{coshy1} and \eqref{coshy3},
we obtain the inequality
\begin{equation}  \label{coshy13}
|\nabla^2f|^2\geq\frac{1}{4n}\Big[(\triangle
f)^2+\sum_{s=1}^{3}[g(\nabla^2f,\omega_s)]^2\Big].
\end{equation}
Finally, using the non-negativity of the $P$-function  for $n>1$, see
\eqref{pppp}, we obtain from \eqref{intin} the desired estimate
\begin{equation*}
\lambda\geq\frac{n}{n+2}k_0.
\end{equation*}
\end{proof}

\begin{cor}
\label{imp} If the case of equality in Theorem~\ref{main1} holds,
i.e., we have
\begin{equation*}  \label{eq1}
\lambda = \frac{n}{n+2}k_0, \qquad \triangle f=\frac{n}{n+2}k_0 f,
\end{equation*}
then the horizontal Hessian of the eigenfunction $f$ is given by \eqref{eq07}%
.
\end{cor}

\begin{proof}
The result follows from \eqref{intin1}, \eqref{condm-app} and \eqref{coshy13}
which asserts that the equalities in \eqref{coshy1} and \eqref{coshy3} must
hold, which imply  \eqref{eq07} .
\end{proof}

\end{document}